\documentclass[usletter,reqno]{amsart}
\usepackage{verbatim}
\usepackage{amsmath}
\usepackage{latexsym}
\usepackage{yfonts}
\usepackage{mathrsfs}
\usepackage{color}
\usepackage{hyperref}

\usepackage{amsfonts}
\usepackage{amsmath}
\usepackage{amssymb}
\usepackage{amsthm}
\usepackage{graphicx}
\usepackage{bm}
\usepackage{bbm}

\newcommand{\be}{\begin{eqnarray} \begin{aligned}}
\newcommand{\ee}{\end{aligned} \end{eqnarray} }
\newcommand{\benn}{\begin{eqnarray*} \begin{aligned}}
\newcommand{\eenn}{\end{aligned} \end{eqnarray*} }

\makeatletter
\def\l@section{\@tocline{1}{0pt}{1pc}{}{}}
\def\l@subsection{\@tocline{2}{0pt}{1pc}{4.6em}{}}
\def\l@subsubsection{\@tocline{3}{0pt}{1pc}{7.6em}{}}
\renewcommand{\tocsection}[3]{%
\indentlabel{\@ifnotempty{#2}{
	\makebox[1.2em][l]{%
	\ignorespaces #2.\hfill}
	}}#3\dotfill}
\renewcommand{\tocsubsection}[3]{%
  \indentlabel{\@ifnotempty{#2}{\hspace*{1.9em}\makebox[2em][l]{%
    \ignorespaces #2.}}}#3\dotfill}
\renewcommand{\tocsubsubsection}[3]{%
  \indentlabel{\@ifnotempty{#2}{\hspace*{3.9em}\makebox[2.8em][l]{%
    \ignorespaces #2.\hfill}}}#3\dotfill}
\makeatother

\newcommand{\bc}{\begin{center}}
\newcommand{\ec}{\end{center}}

\newcommand{\id}{\mathsf{I}}

\newcommand{\tr}{\mathop{\mathsf{tr}}\nolimits}

\newtheorem{theorem}{Theorem}[section]

\newtheorem{lemma}[theorem]{Lemma}
\newtheorem{claim}[theorem]{Claim}

\newtheorem{corollary}[theorem]{Corollary}
\newtheorem{example}[theorem]{Example}

\newcommand{\hil}{\mathcal{H}}

\usepackage{amsfonts}
\def\Real{\mathbb{R}}
\def\Complex{\mathbb{C}}
\def\Natural{\mathbb{N}}
\def\id{\mathbb{I}}

\def\01{\{0,1\}}

\newcommand{\eps}{\varepsilon}
\newcommand{\ket}[1]{|#1\rangle}
\newcommand{\bra}[1]{\langle#1|}
\newcommand{\outp}[2]{|#1\rangle\langle#2|}
\newcommand{\proj}[1]{|#1\rangle\langle#1|}

\newcommand{\inp}[2]{\langle{#1}|{#2}\rangle} 

\newcommand{\End}{{\rm End}}

\newcommand{\Sym}{{\rm Sym}}

\newcommand{\gfrac}[2]{\Gamma\left(\frac{#1}{#2}\right)}
\newcommand{\gam}[1]{\Gamma\left({#1}\right)}
\newcommand{\ints}{\int_{S^{n-1}}}
\newcommand{\maxsym}{ {\rm MSym}}
\newcommand{\maxsymEll}{ {\rm MSym}\left( (\Real^n)^{\otimes \ell}\right)}
\newcommand{\symEll}{ {\rm Sym}\left( (\Complex^n)^{\otimes \ell}\right)}
\newcommand{\maxsymR}{ {\rm MSym}\left( (\Real^n)^{\otimes a}\right)}

\newcommand{\sphere}{S^{n-1}}
\newcommand{\ltwo}{L^2(S^{n-1})}

\newcommand{\sjm}{S_{jm}^\ell}
\newcommand{\sjmR}{S_{jm}^r}
\newcommand{\sjmt}{S_{j'm'}^\ell}

\newcommand{\x}{x}

\newenvironment{sdp}[2]{
\smallskip
\begin{center}
\begin{tabular}{ll}
#1 & #2\\
subject to
}
{
\end{tabular}
\end{center}
\smallskip
}

\newcommand{\U}{\mbox{U}}

\newcommand{\red}{a}
\newcommand{\redC}{ {a_{\Complex}} }
\newcommand{\redR}{ {a_{\Real}} }
\newcommand{\redP}{ {a'} }
\newcommand{\rem}{b}
\newcommand{\suba}{ {a} }

\begin{document}

\title[Convergence of SDP hierachies for polynomial optimization]{Convergence of SDP hierarchies for polynomial optimization on the hypersphere}

\author[Andrew C. Doherty]{Andrew C. Doherty}
\address{Andrew C. Doherty, Centre for Engineered Quantum Systems, School of Physics, University of Sydney, NSW 2006, Australia}
\email{andrew.doherty@sydney.edu.au}
\author[Stephanie Wehner]{Stephanie Wehner}
\address{Stephanie Wehner, Centre for Quantum Technologies, National University of Singapore, 3 Science Drive 2, 117543, Singapore}
\email{steph@locc.la}
\date{\today}
\thanks{We thank Sabine Burgdorf, Monique Laurent, Pablo Parrilo, Markus Schweighofer and Frank Vallentin for helpful comments on an earlier draft. This research was supported by the National Research Foundation and the Ministry of Education, Singapore, 
as well as the ARC Centre of Excellence in Engineered Quantum Systems (EQuS), project number CE110001013.
Part of this work was done while SW was at IQI, Caltech. SW thanks the University of Sydney for their hospitality.}
\begin{abstract}
We show how to bound the accuracy of a family of semi-definite programming relaxations for the problem of polynomial optimization on the hypersphere. Our method 
is inspired by a set of results from quantum information known as quantum de Finetti theorems. In particular, we prove a de Finetti theorem for a special class of real symmetric
matrices to establish the existence of approximate representing measures for moment matrix relaxations. 
\end{abstract}
\maketitle

\tableofcontents

\section{Introduction}
In this paper we will discuss the solution of polynomial optimisation problems of the following form: \emph{What is the maximum value $\nu$ of a polynomial $T(x) = T(x_1,\ldots,x_n)$ in $n$ real variables evaluated on the hypersphere $S^{n-1}$?}
More precisely, our aim is to solve the optimization problem
\begin{sdp}{maximize}{$T(x)$}
&$\|x\|_2 = 1$\ ,
\end{sdp}
with $\|x\|_2^2 = \sum x_j^2$. 
Optimizing polynomials over the sphere or other compact convex sets $K$ has a large number of applications ranging from 
material sciences~\cite{soare:metal}, quantum mechanics~\cite{andrew:separable}, numerical linear algebra, and signal processing 
to portfolio optimization (see~\cite{HLpaper} for an extensive list of applications and references).
Examples of such problems within theoretical computer science itself include finding the maximum stable set ($K$ being the simplex)
or maximum cut ($K$ being the Boolean cube) in graphs (see e.g.~\cite{deklerk} for a survey).

Not surprisingly, not all $K$ pose an equal challenge when it comes to optimization. For example, when $K$ is the simplex it has been known for 
some time that one can approximate the optimum of $T(x)$ over $K$ in polynomial time if $T$ is quadratic~\cite{bomzedeKlerk}, or more generally if $T$
is a homogeneous polynomial of a fixed degree~\cite{parrilo:simplex} (see also~\cite{deklerk}). 
The case of optimizing over the sphere, however, remains poorly understood in comparison~\cite{deklerk}. Whereas the problem is easy if $T$ has degree at most $2$~\cite{deklerk},
it is known that already for degree $3$ the problem of optimizing a polynomial over the sphere is NP-hard~\cite{nesterov}.
This leads to the question of whether polynomial time approximation algorithms do exist. 
For example, it was questioned in~\cite{parrilo:simplex} whether there exists a polynomial time approximation algorithm for homogeneous polynomials over the
sphere. Adhoc approximation algorithms are known for a variety of special cases, for example for optimizing biquadratic polynomials~\cite{ling:quartic}, 
or for optimizing cubic polynomials~\cite{zhang:cubic}.
Recently, a poly time approximation algorithm involving random sampling for optimization of homogeneous polynomials over the sphere was proposed 
in~\cite{HLpaper}. 

There is a general approach to solve \emph{any} such problem approximately by means of a hierarchy of semidefinite programs (SDP)~\cite{parrilo, lasserre} 
(see~\cite{monique:survey} for an excellent survey). 
Very roughly, each level of this SDP hierarchy is of increasing size, and thus more difficult to solve computationally. Yet, solving level $\ell$ of the hierarchy yields an approximation $\nu_\ell$ that improves for increasing $\ell$ and is generally known to converge to $\nu$.
However, results that bound the quality of the approximation for finite $\ell$ are relatively rare. For optimizing over the sphere, using results of Reznik~\cite{reznik},
Faybusovich~\cite{fayb} derived approximation guarantees for the SDP hierarchy 
when optimizing over the sphere, which lead to a meaningful approximation for similar values of
$\ell$ (see Section~\ref{sec:using} for a comparison). However, our approach yields rather new insights and an \emph{explicit} representing measure for every $\ell$.
Results are also known for the special
case that the number of variables $n$ is at least twice the degree of $T(x)$~\cite{nie:nbigger2d}. In the computer science literature, results are also known 
for the special case when the polynomial is of a certain form and has degree $4$~\cite{fernando:hierarchy}.

\subsection{Result}

Here, we prove convergence of a hierarchy of semidefinite programs for optimizing polynomials $T$ on the sphere. More precisely, we show that
for any homogeneous polynomial $T$ 
the solution $\nu_\ell$ of the SDP at level $\ell$
approximates the true optimimum $\nu$ as
\begin{align}\label{eq:expBound}
\nu_\ell - \eps(n,a,\ell)\nu \leq \nu \leq \nu_\ell\ ,
\end{align}
where the relative error is given by
\begin{align}\label{eq:smallErr}
\eps(n,a,\ell) =  \frac{4a^2\left(a + \frac{n}{2} + 1\right)}{2\ell + 1} 
\end{align}
where $a = \lceil d/2 \rceil$.
Since any polynomial of even degree at most $d$~\footnote{I.e., each monomial $t_m(x)$ in $T(x) = \sum_m t_m(x)$ has even degree.} 
on the sphere can be turned into a homogeneous polynomial of degree exactly $d$ 
our result also applies
to such polynomials.
To our best knowledge, these are the first general results about convergence of the SDP hierarchies~\cite{parrilo,lasserre} when optimizing over the sphere that yield explicit representing measures. 
Note that in polynomial optimization the degree of the polynomial is generally regarded as a constant, 
and we care about the performance of the algorithm in terms of the number of variables $n$~\cite{deklerk}.  These bounds do not imply a polynomial time approximation algorithm for optimization over the sphere. However, we emphasize that in contrast to other results on the convergence of the hierarchy of semidefinite programming relaxations 
our results hold for \emph{any} homogeneous $T(x)$, and our approach yields an explicit approximating representing measure at any level $\ell$.

To prove our results we employ rather different techniques than have been used in the past, which can in principle be extended to derive optimization guarantees also for optimizing over other sets $K$ as well. In particular, our techniques are inspired by quantum information (see Section~\ref{sec:qinfo} for a short introduction). First, we use a slightly different way of representing polynomials as matrices which has very nice symmetry properties. Second, we prove 
a so-called de Finetti theorem for real valued matrices that are \emph{maximally symmetric}. This theorem forms the heart
of our convergence result.
In the language of quantum information theory,
a matrix 
$M \in \Sym((\Real^n)^{\otimes \ell})$
in the symmetric subspace of $(\Real^n)^{\otimes \ell}$ satisfies $\pi M = M$ for all $\pi \in S_\ell$ and is often called Bose symmetric. We term such matrices {\it maximally symmetric} if moreover $M$ is invariant under partial transposes of any subset of the $\ell$ systems (i.e., $M^{\Gamma_S} = M$ for all $S \subseteq \{1,\ldots,\ell\}$, where $\Gamma_S$ denotes
the partial transpose of the subsystems contained in $S$). 
In particular, our de Finetti theorem yields explicit probability distributions over the hypersphere with a moment matrix close to the given maximally symmetric matrix $M$. This approximate representing measure for $M$ allows us to place a lower bound on the maximum value of the polynomial $T$ at each level of the hierarchy. 

\begin{theorem}[Real valued de Finetti theorem]
	Let $n,a,\ell \in \Natural$ with $n \geq 3$ and $a < \ell$. Let $M \in \maxsymEll$ be a maximally symmetric matrix that is a state (i.e., $\tr(M) = 1$ and $M \geq 0$), and let
	$Q_M(x) = \bra{x}^{\otimes \ell} M \ket{x}^{\otimes \ell}$ be its $Q$-representation. Define
	\begin{align}
		\tilde{M}_a := \left(\frac{2^{2\ell} \omega_n}{\sqrt{\pi} \omega_{n-1}} \frac{\gam{\ell + 1}\gam{\ell + \frac{n}{2}}}{\gfrac{n-1}{2} \gam{2\ell + 1}}\right)
	\ints Q_M(x) \proj{x}^{\otimes a} dx \ .
	\end{align}
	Then the reduced matrix $M_a := \tr_{\downarrow a}(M)$ is approximated by the matrix $\tilde{M}_a$ as
	\begin{align}
		\left\|M_a - \tilde{M}_a\right\|_{F1} \leq \frac{4a^2\left(a + \frac{n}{2} - 1\right)}{2\ell + n}\ ,
	\end{align}
	with $\|A\|_{F1} := \sup_{\|F\|_{\infty} \leq 1} \tr(Z_F A)$, 
where the maximization is taken over homogeneous polynomials $F$ with polynomial matrix $Z_F$
and $\|F\|_{\infty} = \max_{x \in S^{n-1}} |F(x)|$ is the $p\rightarrow \infty$-norm for functions 
on the sphere. 
	Furthermore, $\tilde{M}_a$ is a state (i.e., $\tr(\tilde{M}_a) \geq 0$, $\tr(\tilde{M}_a) = 1$) and a moment
matrix.
\end{theorem}

At first glance, our de Finetti theorem for real maximally symmetric matrices may seem to contradict the impossibility result of Caves et al.~\cite{caves}
who provided an explicit counterexample indicating that not all real valued matrices in $(\Real^n)^{\otimes \ell}$ can obey a de Finetti theorem. Indeed,
this has led some authors to suggest that de Finetti theorems may not be very useful in the study of convergence of SDP hierarchies~\cite{fernando:hierarchy}. 
However, it turns out that this counterexample is \emph{not} maximally symmetric, and it becomes clear from our proof why this is a crucial condition. 

\subsection{Proof idea}

We give the main idea of the proof of the de Finetti theorem, and explain why existing approaches to proving de Finetti theorems are not useful in the study of maximally
symmetric matrices. The reader unfamiliar with de Finetti theorems may wish to skip ahead and return to this outline when tackling the proof. 
Existing de Finetti theorems~\cite{matthias:definetti,robert:compendious} which lead to an approximation in terms of convex
combinations of product states~\footnote{In contrast to exponential de Finetti theorems~\cite{renato:symmetry}, which however 
are not useful for our application} have made extensive use of representation theory. Most notably they have relied on the requirement that $M$ should act on a system on which the unitary group acts irreducibly, such as the symmetric subspace of $\ell$ copies of some Hilbert space.  Very roughly, the relevant symmetry group in our real case is the orthogonal group whose representation on the symmetric
subspace is no longer irreducible. As such, using irreducibility as in the proof of~\cite{matthias:definetti} is not possible for us.

Instead, we follow a rather different approach inspired by the $P$ and $Q$-repres\-entations in quantum optics, and the results of Laurent in optimization theory~\cite{monique}. 
Very roughly, we proceed along the following steps: First, we show that any maximally symmetric matrix $M$ is fully characterized by a function $P_M(x)$ on the hypersphere, known as the $P$-represention of $M$. Second, we relate the Fourier coefficients of the function $Q_M(x) = \bra{x}^{\otimes \ell} M \ket{x}^{\otimes \ell}$ (the $Q$-representation) 
to the Fourier coefficients of
the function $P_M(x)$. As our whole problem resides on the sphere, the relevant Fourier transform is the one over the sphere in terms of spherical harmonics. This relation
is derived by using a certain convolution theorem in harmonic analysis, known as the Funk-Hecke formula~\cite{mueller1966}. 
Comparing $P_M$ and $Q_M$ inspires us to form an approximation $\tilde{M}_a$ in terms of the $Q$-representation  $Q_M(x)$. Note that whereas $P_M(x)$ can be negative, $Q_M(x)$ is positive since $M \geq 0$.
Third, we show that the low order Fourier coefficients of $P_M(x)$ and $Q_M(x)$ are actually (up to normalization) almost the same. Finally, we show that when comparing the reduced state $M_a$ of $M$ with $\tilde{M}_a$ only the low order Fourier coefficients of $P_M$ and $Q_M$ are relevant -- but these have already been shown to be the almost the same in step 3.

\section{Basic concepts}

Our construction makes essential use of results from quantum information, representation theory, and spherical harmonic functions.  
In Section~\ref{sec:qinfo} we review the necessary ingredients from quantum information. In Section~\ref{sec:repTheory} we give 
some essential background from representation theory. We continue to summarize facts of spherical harmonics in Section~\ref{sec:spherical}. Finally, we define the notion of maximally symmetric matrices in Section~\ref{sec:maxSym}.

\subsection{Quantum information}\label{sec:qinfo}
\subsubsection{Notation}
Throughout, we use notation that is commonly used within quantum information which served as an inspiration for our proof. 
Consider a vector space 
$V = \Complex^n$ or $V = \Real^n$ of dimension $n$~\footnote{All spaces in this note are of finite dimension.}. 
We use $\ket{v} \in V$ to denote a vector in $V$, and $\bra{v}$ to denote its dual vector, the complex conjugate transpose of $\ket{v}$. Let $\End(V)$ denote the set of linear
operators on $V$. For any matrix $M \in \End(V)$ we will use $M^*$ to denote its complex conjugate and $M^\dagger = (M^*)^T$ to denote its complex conjugate transpose (sometimes 
denoted by $M^*$ in other fields). 
For any matrix $M \in \End(V)$, we use $\tr(M)$ to denote the \emph{trace} of $M$, that is, the sum of its diagonal entries.
The identity is denoted by $\id$.

\subsubsection{States, systems and tensor products}
A (quantum) \emph{state} is a matrix $M$ satisfying $M \geq 0$ and $\tr(M) = 1$.
We also consider vectors spaces of the form $V \otimes V$, where $\otimes$ denotes the Kronecker product, also known as the \emph{tensor product}.
When considering the tensor product $V^{\otimes \ell} = V \otimes \ldots \otimes V$ of $\ell$ vector spaces, we will also refer to the individual spaces $V$ as
\emph{systems}. If $\ket{x} \in V$ and $\ket{y} \in V$ then the vector $\ket{x} \otimes \ket{v} \in V \otimes V$. If it is clear that from context that we are considering tensor products of vectors on different systems, we follow the common convention in quantum information
and employ the shorthand
\begin{align}
\ket{xy} := \ket{x} \otimes \ket{y}\ .
\end{align}
When amgiuity is possible, we also use $\ket{x,y} := \ket{xy}$.

\begin{example}
	Consider two $2$-dimensional vectors spaces $V_1 = V_2 = \Complex^2$, where $\{\ket{v_1}, \ket{v_2}\}$ 
	and $\{\ket{\hat{v}_1}, \ket{\hat{v}_2}\}$
	are orthonormal bases for $V_1$ and $V_2$ respectively. 
A basis for the space $V_1 \otimes V_2$ is
given by $\{\ket{v_1\hat{v}_1}, \ket{v_1\hat{v_2}},\ket{v_2\hat{v}_1},\ket{v_2\hat{v}_2}\}$ where for example
\begin{align}
	\ket{v_1 \hat{v}_1} = \ket{v_1} \otimes \ket{\hat{v}_1} =
(x_1 \hat{x}_1,x_1 \hat{x}_2,
x_2 \hat{x}_1,x_2 \hat{x}_2)^T\ ,
\end{align}
with 
$\ket{v_1} = (x_1,x_2)$ and $\ket{\hat{v}_1} = (\hat{x}_1,\hat{x}_2)$.
\end{example}

\subsubsection{Partial trace of matrices}
An important notion we will need later is the so-called partial trace operation. When considering polynomials as matrices we will see that this operation is analogous to
applying the Laplacian to a polynomial.
Consider the tensor product of spaces
$U \otimes W$. For a matrix $M \in \End(U\otimes W)$ we can define $\tr_{W}(M) \in \End(U)$ the  \emph{partial trace} of $M$ \emph{over} $W$ such that
\begin{align}
	\tr\left[(N\otimes I)M\right] = \tr [N \tr_W(M)] \quad \forall N\in \End(U) \ ,
\end{align}
We can write an explicit formula for the components of $\tr_{W}(M) $. Let $\{\ket{u}\}_u$ and $\{\ket{w}\}_w$ be orthonormal bases for $U$ and $W$ respectively.
We can express any matrix $M \in \End(U\otimes W)$ in terms of this basis as
\begin{align}
	M = 
\sum_{u,u',w,w'} m_{u,u',w,w'} \outp{u}{u'} \otimes \outp{w}{w'}\ ,
\end{align}
for some coefficients $m_{u,u',w,w'} \in \Complex$. 
The \emph{partial trace} of $M$ \emph{over} $W$ can be shown to be as follows
\begin{align}\label{def:partialTrace}
	\tr_{W}(M) &:= \sum_{u,u',w,w'} m_{u,u',w,w'} \outp{u}{u'} \otimes \tr\left(\outp{w}{w'}\right)\\
&= 
\sum_{u,u'} \left(\sum_{w} m_{u,u',w,w}\right) \outp{u}{u'}\ .
\end{align}
We also speak of \emph{tracing out} $W$ from $M$ to obtain $\tr_{W}(M) \in \End(U)$.
When $U = V^{\otimes \red}$, $W = V^{\otimes \rem}$, and $\ell = \red + \rem$ we also write the partial trace as
\begin{align}
	\tr_{\downarrow a}(M) := \tr_{\red+1,\ldots,\ell}(M) := \tr_{W}(M)\ ,
\end{align}
to emphasize the number of systems we trace out.
It is sometimes useful to note that the partial trace can also be written as
\begin{align}
	\tr_W(M) = \sum_{w} (\id \otimes \bra{w})M(\id \otimes \ket{w})\ ,
\end{align}
where the sum is taken over basis vectors $\ket{w}$ of an orthonormal basis on $W$.

Of particular interest to us will be the partial trace of Hermitian matrices $M$ such that $M \geq 0$ and $\tr(M) = 1$, i.e., states. 
When tracing out a number of systems of a state $M \in \End(V^{\otimes \ell})$
we also call $M_{\suba} := \tr_{\red+1,\ldots,\ell}(M)$ a \emph{reduced state} of $M$.
It is clear from the definition of the partial trace above that $\tr(M_\suba) = 1$. 
Similarly, note that $M_\suba \geq 0$ for any state $M$ since $\bra{v}M_\suba\ket{v} = \tr\left((\proj{v} \otimes \id^{\otimes \rem})M\right)$, 
where $\id$ is the identity on $V$ and $\rem = \ell - \red$.

\begin{example}
To get some intuition about the partial trace operation, consider the following two examples. First, consider a state $M \in \End(V^{\otimes 2})$ where $M = M_1 \otimes M_2$ and $M_1,M_2\in \End(V)$ being states. When tracing out the second system we simply obtain
$\tr_2(M) = M_1 \tr(M_2) = M_1$. Second, let us consider a state $M$ that cannot be written as a tensor product of state on the individual
systems. Consider $M = \proj{\Psi}$ for $\ket{\Psi} = (\ket{11} + \ket{22})/\sqrt{2}$ where $\{\ket{1},\ket{2}\}$ forms an orthonormal basis
for $V$ with dimension $\dim(V) = 2$. As a matrix,
\begin{align}
M &= \frac{1}{2}\left(\begin{array}{cccc} 1 & 0 & 0 & 1\\ 
0 & 0 & 0 & 0 \\
0 & 0 & 0 & 0 \\
1 & 0 & 0 & 1 \end{array} \right)\\
&= 
\frac{1}{2}\left(\proj{1} \otimes \proj{1} + \outp{1}{2} \otimes \outp{1}{2}
+ \outp{2}{1} \otimes \outp{2}{1} + \proj{2} \otimes \proj{2}\right)\ .
\end{align}
By the definition of the partial trace~\eqref{def:partialTrace} we have
\begin{align}
\tr_2\left(M\right) &=
\frac{1}{2}
\left(\proj{1} \otimes \tr\left(\proj{1}\right) + \outp{1}{2} \otimes \tr\left(\outp{1}{2}\right)\right.\\
&\qquad \left. + \outp{2}{1} \otimes \tr\left(\outp{2}{1}\right) + \proj{2} \otimes \tr\left(\proj{2}\right)\right)\nonumber\\
&=\frac{1}{2}\left(\proj{1} + \proj{2}\right)\\
&= \frac{\id}{2}\ .
\end{align}
\end{example}

\subsubsection{Matrices as vectors}\label{sec:vectorForm}

When considering the case of real valued variables, it will be useful to note that any matrix $M \in \End(V^{\otimes \ell})$ can be identified uniquely with a vector $\ket{M} \in V^{\otimes 2\ell}$. Let $\{\ket{i}\}_{i=1}^{\dim(V)}$ denote the standard 
orthonormal basis of $V$. Note that the set of matrices $\outp{i}{j}$ forms a basis for $\End(V)$ in which we can express $M$ as
\begin{align}
M = \sum_{ij} m_{ij} \outp{i}{j}\ ,
\end{align}
for some coefficients $m_{ij} \in \Real$ (or $m_{ij} \in \Complex$ for $V = \Complex^n$). Relative to our chosen basis, we can hence define an isomorphism between matrices and vectors given as
\begin{align}
\ket{M} := \sum_{ij} m_{ij} \ket{i}(\ket{j})^*\ .
\end{align}
Throughout, when concerned with matrices $M$ we will use $\ket{M}$ to denote its vectorized form. It will be useful to note that for this mapping we have for any $\ket{x} \in V$ 
\begin{align}
\bra{x} M \ket{x} = \bra{x}\otimes(\bra{x})^* \ket{M}\ .
\end{align}

\subsubsection{Norms and a distance measure}

Finally, we will use two particular norms on the set of operators $\End(V)$.
The first is the \emph{operator norm} which for Hermitian matrices $N$ is simply given by $\|N\|_\infty = \lambda_{\rm max}(|N|)$, where $\lambda_{\max}(|N|)$
denotes the largest eigenvalue of $|N|$. We will also need its dual norm, the L1-norm, which for any matrix $M \in \End(V)$ can be expressed as
\begin{align}
	\|M\|_1 = \sup_{\substack{\|N\|_\infty \leq 1}} \tr(NM)\ .
\end{align}
The L1-norm leads to a very useful distance measure on the subset of states in $\End(V)$. In particular, for two quantum states 
$M_1$ and $M_2$
the following quantity forms a distance
\begin{align}
	D(M_1,M_2) := \|M_1 - M_2\|_1\ ,\label{eq:traceDist}
\end{align}
also known as the \emph{trace distance}.
Immediately from its definition it is clear that the trace distance has a useful property.
Suppose two states $M_1,M_2 \in \End(V)$ are $\eps$-close in trace distance, i.e., $D(M_1,M_2) \leq \eps$.
We then have for any matrix $N$ that
\begin{align}\label{eq:nearlySameValue}
	\left|\tr\left(NM_1\right) - \tr\left(NM_2\right)\right| \leq \eps \|N\|_{\infty}\ .
\end{align}

\subsection{Representation theory and the symmetric subspace}\label{sec:repTheory}

We will make extensive use of tensor product spaces and various elements of representation theory of the unitary and orthogonal groups (see e.g.~\cite{goodman:book} for an in-depth introduction).

\subsubsection{Unitary and symmetric groups}

The space $V^{\otimes \ell}$ with $V = \Complex^n$ carries a natural representation of
the group of permutations of $\ell$ objects $S_\ell$ and the unitary group $\U(n)$ with dimension $n = \dim(V)$. 
Consider the representation of $S_\ell$ on $V^{\otimes \ell}$ given by
\begin{align}
	\pi\left(\ket{x^{(1)}} \otimes \ldots \otimes \ket{x^{(\ell)}}\right) = \ket{x^{(\pi(1))}} \otimes \ldots \otimes \ket{x^{(\pi(\ell))}}\ , 
\end{align}
for $\pi \in S_\ell$. Similarly, consider a representation of $\U(n)$ given by the tensor products of the defining representation of $\U(n)$ on $V$ as
\begin{align}
	g\left(\ket{x^{(1)}} \otimes \ldots \otimes \ket{x^{(\ell)}}\right) = g\ket{x^{(1)}} \otimes \ldots \otimes g\ket{x^{(\ell)}}\ , 
\end{align}
for $g \in \U(n)$. Clearly the actions, of $S_\ell$ and $\U(n)$ commute so that this representation on $V^{\otimes \ell}$ will break up into irreducible representations of the direct product group $S_\ell\times \U(n)$. 


\subsubsection{Symmetric subspace}

The symmetric subspace is the subspace of $V^{\otimes \ell}$ corresponding to the trivial representation of $S_\ell$. The projector onto the symmetric subspace
is given by
\begin{align}
	\Pi_\ell = \frac{1}{\ell!} \sum_{\pi \in S_\ell} \pi\ .
\end{align}
As a result we may write the symmetric subspace as
\begin{align}
	\Sym(V^{\otimes \ell}) = \{M \in \End(V^{\otimes \ell})\mid \Pi_{\ell} M \Pi_{\ell} = M\}\ .
\end{align}
It is a standard result that the restriction of our representation $g^{\otimes \ell}$ of $\U(n)$ to the symmetric subspace is irreducible. Note that any vector of the form 
$\ket{z} := \ket{x}^{\otimes \ell}$ obeys $\pi\ket{z} = \ket{z}$ for all $\pi$ and thus
$\Pi_\ell\ket{z} = \ket{z}$. So $\ket{z}$ lies entirely the symmetric subspace $\Sym(V^{\otimes \ell})$. 


\subsubsection{Number state basis}
A convenient basis for the symmetric subspace is given by the so-called \emph{number state basis} $\{\ket{\textbf{i}}\}_{\textbf{i}}$ 
with
\begin{align}\label{eq:numberState}
	\ket{\textbf{i}} 
	:= \frac{1}{\sqrt{\ell!\textbf{i}!}} \sum_{\pi \in S_\ell} \pi\left(\ket{1}^{\otimes i_1} \otimes \ldots \otimes \ket{n}^{\otimes i_n}\right)\ ,
\end{align}
where $\textbf{i} := (i_1,\ldots,i_n) \mbox{ with } \sum_j i_j = \ell$, $\textbf{i}! = i_1!\ldots i_n!$, and $\{\ket{1},\ldots,\ket{n}\}$ is an orthonormal basis for $V$.
Note that the projector into the symmetric subspace can also be expressed in the number state basis as $\Pi = \sum_{\textbf{i}} \proj{\textbf{i}}$. 
Furthemore, note that the set of vectors
$\{\pi \ket{v_{\textbf{k}}}\mid \ket{v_{\textbf{k}}} := 
\ket{x_1}^{\otimes k_1}\otimes \ldots \otimes \ket{x_n}^{\otimes k_n} \mbox{ with } \pi \in S_\ell\}$ 
spans $(\Complex^n)^{\otimes \ell}$. 
Hence any matrix $M \in \End((\Complex^n)^{\otimes \ell})$ can be written in the form
\begin{align}
M = \sum_{\substack{\pi,\sigma \in S_\ell\\ \textbf{k},\textbf{m}}} 
f_{\pi,\sigma}^{\textbf{k},\textbf{m}}
\pi\ket{v_{\textbf{k}}}\bra{v_{\textbf{m}}} \sigma\ ,
\end{align}
for some coefficients $f_{\pi,\sigma}^{\textbf{k},\textbf{m}}$. 
It will be useful to understand what matrices written in this form look like 
when projected into the symmetric subspace. 
We can compute $\Pi_\ell M \Pi_\ell$ by computing the terms
\begin{align}
\label{eq:innerprod}
\bra{\textbf{i}} \pi \ket{v_{\textbf{k}}} \bra{v_{\textbf{m}}} \sigma\ket{\textbf{j}} &= 
\inp{\textbf{i}}{v_{\textbf{k}}} \inp{v_{\textbf{m}}}{\textbf{j}} =
c_{\textbf{i}} c_{\textbf{j}}\ \textbf{i}! \textbf{j}!\ \delta_{\textbf{i},\textbf{k}} \delta_{\textbf{j},\textbf{m}}\ ,
\end{align}
where we used the shorthand
\begin{align}\label{eq:numberShort}
	c_{\textbf{i}} := \frac{1}{\sqrt{\ell! \textbf{i}!}}\ .
\end{align}
It is also interesting to consider what happens to the matrix $\outp{\textbf{i}}{\textbf{j}}$ 
if we trace out one system from $(\Complex^n)^{\otimes \ell}$. 
We have 
\begin{align}
&\tr_{1}\left[\outp{\textbf{i}}{\textbf{j}}\right]
=c_{\textbf{i}} c_{\textbf{j}}\\
&\sum_{\pi_1,\pi_2 \in S_{\ell-1}}
\left(\sum_{m_1,m_2 \in \{1,\ldots,\ell\}} \tr_{1}\left[(\ell,m_1)\ket{1}^{\otimes i_1}\ldots\ket{n}^{\otimes i_n}
\bra{1}^{\otimes j_1} \ldots \bra{n}^{\otimes j_n} (\ell,m_2)\right]\right) \nonumber \\
&
=c_{\textbf{i}} c_{\textbf{j}} \sum_{\pi_1,\pi_2 \in S_{\ell-1}}
\sum_{t=1}^n i_t j_t 
\ket{1}^{\otimes i_1}\ldots \ket{t}^{\otimes i_t - 1}\ldots \ket{n}^{\otimes i_n}
\bra{1}^{\otimes j_1}\ldots \bra{t}^{\otimes j_t - 1}\ldots \bra{n}^{\otimes j_n}\\
&=\frac{1}{\ell} \sum_{t=1}^n \sqrt{i_t j_t} \outp{\textbf{i} - e_t}{\textbf{j}-e_t}\ ,
\label{eq:ptrace}
\end{align}
with $\textbf{i} - e_t  = (i_1,\ldots,i_t-1,\ldots,i_n)$, where the first equality stems from the fact that we can write any permutation over $\ell$ elements as a swap between the last and any of the others, followed by a permutaton over only $\ell-1$ elements.


\subsubsection{Orthogonal group}

While the unitary group will be sufficient to deal with the case of polynomials on the complex projective space, in the case of real polynomials on the sphere we will also need to deal with representations of the orthogonal group. As before the space  $V^{\otimes \ell}$  carries a natural representation of the orthogonal group as follows
\begin{align}\label{eq:Orep}
	g\left(\ket{x^{(1)}} \otimes \ldots \ket{x^{(\ell)}}\right) = g\ket{x^{(1)}} \otimes \ldots \otimes g\ket{x^{(\ell)}}\ .
\end{align}
for all $g\in O(n)$. $O(n)$ is a subgroup of $U(n)$ and again this representation commutes with the representation of $S_\ell$ that permutes copies of $V$. Once again we will be able to restrict our attention to the symmetric subspace $\Sym(V^{\otimes \ell})$. However the symmetric subspace is no longer irreducible as a representation of $O(n)$ but rather decomposes into into a number of irreducibles  (see e.g.~\cite[Chapter V]{weyl:book} or~\cite[Chapters 17.3,19.5]{fulton:rep}).

%

\subsection{Polynomials}\label{sec:poly}

\subsubsection{Real variables}\label{sec:evenoddDegree}
We will first consider the case of real polynomials. For a set of real variables $x\in \Real^n$ a general monomial of order $d$ can be written as $x_1^{i_1}...x_n^{i_n}=x^{\bf i}$ with $|{\bf i}|=d$. A general polynomial degree at most $d$ can therefore be written as
\begin{align}
\label{eq:polydefnreal}
	T({\x}) = \sum_{
	|{\bf i} | \leq d} 
	\alpha_{\bf i}{\x^{{\bf i}}}\ ,
\end{align}
for some coefficients $\alpha_{\bf i}\in \Real$. We will largely be interested in evaluating these polynomials on the hypersphere 
where the variables $x_1,\ldots,x_n$ satisfy
 \begin{align}\label{eq:realRdef}
	 r(x)^2 :=\sum_{i=1}^n x_i^2=1\ .
\end{align}
Let us now first consider the case where all terms in $T(x)$ have even order, i.e., all monomials in $T(x)$ have even degree. 
We can then use the fact that we are evaluating polynomials on the sphere to write $T$ as a sum of monomials of degree \emph{exactly} $d$ since
for any ${\bf i}$ with $ |{\bf  i}|  =: d' \leq d$ we have
\begin{equation}
	\alpha_{\bf i}{ x^{i}} =\alpha_{\bf i}{ x^{i}} r(x)^{d-d'} \ ,
\end{equation}
which is a homogeneous polynomial of degree $d$ when $d$ is even. Note that this turns any such polynomial into a homogeneous polynomial of even degree.

Let us now consider homogeneous polynomials $T(x)$ of odd degree $d = 2a-1$. As we show in the appendix (Lemma~\ref{lem:turnIntoEven}), there exists an explicit constant $\gamma(a)$ such
that 
\begin{align}
\max_{x' \in S^n} T'(x')  = \gamma(a) \max_{x \in S^{n-1}} T(x)\ ,
\end{align}
where $T'(x')$ is a homogeneous polynomial of even degree. Solving the l.h.s. thus allow us to obtain the solution of the r.h.s.
Throughout, we will thus always assume that $T$ has been transformed into a homogeneous polynomial of even degree exactly $2a$.

We define the Laplacian in the usual way as follows
\begin{equation}
\Delta T(x)=\sum_i\frac{\partial^2}{\partial x_i^2} T(x)\,
\end{equation}
which maps a polynomial of degree $d$ to a polynomial of degree $d-2$. {\it Harmonic polynomials} are those for which $\Delta T(x)=0$.
Real polynomials carry a natural representation of the orthogonal group as follows
\begin{equation}
T(x)\rightarrow L_g[T](x) = T(g^{-1}x) \quad \forall g\in O(n).
\end{equation}
While this is not an irreducible representation, it is well known that when restricted to the harmonic polynomials it becomes irreducible~\cite{weyl:book,fulton:rep,takeuchi,mueller}. 

\subsubsection{Complex variables}
We will also consider polynomials $T(x)$ in complex variables $x = (x_1,\ldots,x_n)\in \Complex^n$. 
This is because our techniques originate from quantum information and turn out to be somewhat simpler for the case of complex variables.
In this case it is natural to consider polynomials in $x$ and $x^*$ with monomials of the form  $x^{*i_1}_1x^{*i_2}_2\ldots x^{*i_n}_n x_1^{j_1}x_2^{j_2}\ldots x_n^{j_n}=x^{*{\bf i}}x^{\bf j}$. When $|{\bf i}|=\redC$ and $|{\bf j}|=\redR$ this is a monomial of order $(\redC,\redR)$. 
A general polynomial will then be of the form $T(x,x^*)=\sum_{{\bf i,j}}\alpha^{\bf i}_{\bf j}x^{*\bf i}x^{\bf j}$. We will be interested in evaluating these polynomials on the complex projective space $\Complex P^{n-1}$ of points $x\in \Complex^n$ where $\beta x$ is identified with $x$ for all $\beta\neq 0 \in \Complex$. The monomial $x^{*{\bf i}}x^{\bf j}$ is a function on  $\Complex P^{n-1}$ only if 
$|{\bf i}|=|{\bf j}|=\red$ and we call this a monomial of {\it symmetric degree} $d=2\red$. 
Moreover for the maximization to make sense we need $T(x,x^*)$ to be real for all $x$ and we call such polynomials Hermitian. 
A polynomial $T$ of order $(\red,\red)$ is Hermitian if $T$ can be expanded in monomials of symmetric degree at most $d$ as follows
\begin{align}
\label{eq:polydefnc}
	T({\x},x^*) = \sum_{
	|{\bf i} |= |{\bf j}| \leq a} 
	\alpha^{\bf i}_{\bf j}{\x^{*{\bf i}}\x^{{\bf j}}}\ ,
\end{align}
for some coefficients $\alpha^{\bf i}_{\bf j}\in \Complex$ satisfying $(\alpha^{\bf i}_{\bf j})^*=\alpha^{\bf j}_{\bf i}$.

The polynomial function $T(x,x^*)$ is not invariant under $x\rightarrow cx$ and is therefore not yet a function on $\Complex P^{n-1}$. To resolve this issue we define a normalised vector $\hat{x}=x/\|x\|_2$ which is invariant under $x\rightarrow cx$ for positive real $c$. As described above we have already restricted attention to Hermitian polynomials of symmetric degree that are invariant when $c=\exp(i\theta)$ for any angle $\theta$. As a result we may define the following function on  $\Complex P^{n-1}$
\begin{equation}
T_{\Complex P^{n-1}}(x,x^*)=T(\hat{x},\hat{x}^*).
\end{equation}
When we refer to optimising a polynomial function over complex projective space we mean that we optimize $T_{\Complex P^{n-1}}(x,x^*)$, which is properly a polynomial function of $\hat{x}$ rather than $x$. However this distinction will usually be unimportant

As in the real case we can use the fact that we may embed $\Complex P^{n-1}$ in the set of normalised vectors $|x\rangle $ such that 
\begin{align}
r(x)^2 =|\langle x|x\rangle |^2=\sum_kx_k^*x_k=1\ . 
\end{align}
So we may write $T$ as a sum of a sum of monomials of symmetric degree \emph{exactly} $d$ since
for any ${\bf i}$ and ${\bf j}$ with $ |{\bf  i}| =| {\bf j}| =: \redP \leq \red$ we have
\begin{equation}
	\alpha^{\bf i}_{\bf j}{ x^{\bf *i}x^{\bf j}} =\alpha^{\bf i}_{\bf j}{ x^{\bf *i}x^{\bf j}}r(x)^{2(\red-\redP)}
	 \ ,
\end{equation}
and this is a sum of monomials of symmetric degree $2\red$.
Throughout, we will always assume that $T$ has been transformed in this fashion.

Another natural operation we can perform on a Hermitian polynomial is to act with the Laplacian
\begin{equation}
T(x,x^*)\rightarrow \Delta_C T(x,x^*)=\sum_k\frac{\partial}{\partial x_k^*}\frac{\partial}{\partial x_k}T(x,x^*)
\end{equation}
which maps a Hermitian polynomial of degree $2\red$ to a Hermitian polynomial of degree $2\red-2$. Notice that if we write $x=x_{\Real}+ix_{\Complex}$ in terms of real and imaginary parts then $T(x,x^*)$ is a polynomial in the $2n$ real variables $(x_{\Real},x_{\Complex})^T$ and the Hermitian harmonic polynomials of symmetric degree $2\red$ in $n$ complex variables become a subset of the harmonic polynomials of symmetric degree $d=2\red$ in $2n$ real variables. There is a natural representation of the unitary group $\U(n)\subset O(2n)$ on the Hermitian polynomials of symmetric degree $d$ as follows
\begin{equation}
T(x,x^*)\rightarrow L_g[T](x,x^*) = T(g^{-1}x,gx^*) \quad \forall g\in \U(n).
\end{equation}
As in the real case this is not an irreducible representation but becomes so when restricted to the harmonic Hermitian polynomials of symmetric degree $2\red$~\cite{takeuchi}.

\subsection{Spherical harmonics}\label{sec:spherical}

Our result for real valued variables depends on facts of spherical harmonic functions (see e.g.~\cite{mueller} for an in-depth introduction). It should be noted that these can be defined in a rather general way (see e.g.~\cite{takeuchi}), however here we only require spherical harmonic functions on the real sphere $\sphere \subset \Real^n$ and hence we restrict ourselves to this case.

\subsubsection{Functions on the sphere}

Let us first consider general functions on the sphere $S^{n-1}$. We will use $dx$ to denote the measure on $S^{n-1}$ that is normalized to unity, i.e.
\begin{align}
\ints dx = 1\ .
\end{align}
Let $\ltwo$ denote the space of square-integrable functions on $\sphere$ with respect to the measure 
$dx$~\footnote{A function is square-integrable with respect to $dx$ if $\|f\|_2 = (\ints |f(x)|^2 dx)^{1/2} < \infty$.}.
It is a Hilbert space under the natural inner product 
\begin{align}\label{eq:innerProductFunctions}
(f,g) := \ints f(x) g(x) dx\ .
\end{align}

It is well known (see e.g.~\cite{takeuchi,sharmonics}) that the space of functions can be decomposed as
\begin{align}\label{eq:l2dec}
\ltwo \simeq \bigoplus_{j=0}^{\infty} \hil_j(\sphere)\ ,
\end{align}
where $\hil_j(\sphere)$ is the space of all so-called~\emph{spherical harmonic functions} determined by the restriction of all homogenous harmonic polynomials of degree $j$ on $\Real^n$ to the sphere. The space $\hil_j(\sphere)$ inherits a natural representation of $O(n)$ from the corresponding representation on the harmonic polynomials, and this representation is still irreducible.
For each $\hil_j(\sphere)$ there exists a basis of functions given by the \emph{spherical harmonics} $s_{jm}$. There exist
\begin{align}
N(j,n) = {n + j - 1 \choose j} - {n + j - 3 \choose j - 2}
\end{align}
linearly independent spherical harmonics $s_{jm}$ of degree $j$. 
We choose this basis as in~\cite{mueller} to be orthogonal so that 
\begin{align}\label{eq:sphericalNorm}
	\ints s_{jm}(x) s_{j'm'}(x) dx = \delta_{jj'} \delta_{mm'} \frac{1}{\omega_n}\ ,
\end{align}
where 
\begin{align}\label{eq:omegaDef}
	\omega_n := \frac{2\pi^{n/2}}{\gfrac{n}{2}}
\end{align}
is the surface area of the sphere $S^{n-1}$ when integrating over the usual Cartesian measure in $\Real^n$. 
Furthermore, any spherical harmonic of degree $j$ can be expressed as a linear combination of the $s_{jm}$.

In terms of this basis, any (continuous) function $f \in \ltwo$ converges uniformly to
\begin{align}\label{eq:fourier}
f(x) = \sum_{j=0}^{\infty}\sum_{m=1}^{N(j,n)} f_{jm} s_{jm}(x)\ .
\end{align}
Note that this is the generalization of the Fourier transform on the circle to the sphere. The coefficients $f_{jm}$ are thus also referred to as \emph{Fourier coefficients}. 

Very roughly, the fact that $\ltwo$ admits a decomposition of the form~\eqref{eq:l2dec} can be understood as a consequence of the fact that by the separation of variables theorem for polynomials~\cite{tonthat1976}, any homogenous polynomial of degree $d$ can be written as a sum of terms each of which consists of a 
harmonic polynomial of degree $d-2k$ multiplied by a radical polynomial $r^{2k}(x)$. Restricting to the sphere sets the radical terms to unity and a general polynomial can be expanded as a sum of homogenous harmonic polynomials of degree up to $d$ alone. The fact that any function can be decomposed as~\eqref{eq:fourier} can now be understood as a consequence of Weierstrass' theorem stating that the algebra of all polynomials is dense in $\ltwo$. Intuitively, this means that any function can be approximated by a sum of polynomials of possibly high degree. Considering the origins of said composition, however, it is instructive to note that when $f$ in~\eqref{eq:fourier} is a polynomial of degree up to $d$, we will also only observe terms up to $j \leq d$. 

\subsubsection{Funk-Hecke formula}
We will need one important property of the spherical harmonic functions. Consider a spherical harmonic $f_j\in \hil_j(\sphere)$ and $|y\rangle \in S^{n-1}$, then
\begin{equation}
\label{eq:FunkHecke}
\int_{S^{n-1}} (\langle x|y\rangle)^{2\ell} f_j(x) dx = \frac{\omega_{n-1}}{\omega_n}\lambda (n,\ell,j) f_j(y)\ ,
\end{equation}
where $\lambda(n,\ell,j)$ is a normalization factor as defined in Lemma~\ref{lem:simplifyLambda}.
This equation is a special case of the Funk-Hecke formula~\cite{mueller} applied to the continuous function $(\langle x|y\rangle)^{2\ell} $. 
Note that the appearance of the factor $1/\omega_n$ arises from converting~\cite[Theorem 6]{mueller} to account for the fact that, in contrast to M{\"{u}}ller, 
we are using a normalised measure on the hypersphere.  
The Funk-Hecke formula can be understood as providing the normalization factor of Schur's lemma for the representation space $\hil_j(\sphere)$\footnote{Notice that this convolution integral commutes with the action of $O(n)$ on $\hil_j(\sphere)$.}.

\subsubsection{Norms}
The set of functions on the sphere admits a set of so-called $p$-norms. Here we will merely need the \emph{$\infty$-norm of a function} $F$ on the sphere
given by
\begin{align}
	\|F\|_\infty := \max_{x \in S^{n-1}} |F(x)|\ .
\end{align}

\subsection{Maximally symmetric matrices}\label{sec:maxSym}

As noted above we will often need to discuss matrices on the symmetric subspace $M \in \End(\Sym(V^{\otimes \ell}))$ for which $\Pi_\ell M \Pi_\ell =M$. It can be shown that all matrices on the symmetric subspace in fact satisfy $\Pi_\ell M=M$ and such matrices are sometimes termed {\it Bose symmetric} in the quantum information theory literature. However, in the case of real polynomials we will need to consider a set of matrices with an even stronger set of symmetries.

\subsubsection{Definition}
We will call $M$ a \emph{maximally symmetric matrix} if 
\begin{align}
\pi \ket{M} = \ket{M} \mbox{ for all }\pi \in S_{2\ell}\ , 
\end{align}
where $\ket{M}$ is the vector form of the matrix $M$ as defined in Section~\ref{sec:vectorForm} and $\pi$ is a permutation of the copies
of $V^{\otimes 2\ell}$ as disussed in Section~\ref{sec:repTheory}. Another way of phrasing this condition is that $\Pi_{2\ell}|M\rangle = |M\rangle$.
This condition also implies that $\pi M = M$ for all $\pi \in S_\ell$ and so $M$ is indeed a matrix on $\Sym(V^{\otimes \ell})$. 
For the reader from quantum information theory, we note that $M$ is not only Bose-symmetric but also  {\it PPT-symmmetric}, i.e., 
invariant under partial transposes of any subset of the $\ell$ systems. It is also helpful to note that any maximally symmetric matrix is real, since its definition also implies that $M = M^T$.
Throughout, we will use
\begin{align}
\maxsym((\Real^n)^{\otimes \ell}) := \left\{ M \in \End((\Real^n)^{\otimes \ell} \mid \Pi_{2\ell} \ket{M} = \ket{M} \right \}\ 
\end{align}
to denote the subspace of maximally symmetric matrices of $\End((\Real^n)^{\otimes \ell})$.


\subsubsection{Spherical harmonic matrices $S_{jm}^{\ell}$}\label{app:prop}

For every spherical harmonic $s_{jm}$ we associate a maximally symmetric matrix, a \emph{spherical harmonic matrix}, defined as
\begin{align}\label{eq:defineSJM}
\sjm := \ints s_{jm}(x) \proj{x}^{\otimes \ell} dx\ .
\end{align}
Clearly we have $\sjm \in \maxsymEll$. 
We will need to establish several properties of the spherical harmonic matrices $\sjm$ as defined in~\eqref{eq:defineSJM}. The first lemma shows that they are in fact orthogonal with respect to the Hilbert-Schmidt inner product. It also establishes their normalization. In essence, our first lemma is an easy consequence of the Funk-Hecke formula given
in~\eqref{eq:FunkHecke}.

\begin{lemma}\label{lem:sjmTrace}
	Let $\ell, j, n \in \Natural$ with $n \geq 3$ and $j \leq 2\ell$, 
	then for matrices $\sjm$ and $\sjmt \in \maxsym( (\Real^n)^{\otimes \ell})$ as defined in~\eqref{eq:defineSJM}, we have
	\begin{align}
		\tr\left(\sjm \sjmt\right) = 
		\delta_{j j'} \delta_{m m'} \frac{w_{n-1}}{w_n^2}\ \lambda(n,\ell,j)\ ,
	\end{align}
	where the normalization factor $\lambda(n,\ell,j)$ is defined in Lemma~\ref{lem:simplifyLambda}.
\end{lemma}
\begin{proof}
	Using the fact that the trace is linear and cyclic we have
	\begin{align}
		\tr\left(\sjm \sjmt\right) 
		&= \ints \ints s_{jm}(x) s_{j'm'}(z)\tr\left(\proj{x}^{\otimes \ell} \proj{z}^{\otimes \ell}\right) dx dz\\
		&=  \ints \ints s_{jm}(x) s_{j'm'}(z) (\inp{x}{z})^{2\ell} dx dz\ .\label{eq:traceExpand}
	\end{align}
	Since $t^{2\ell}$ is continuous for $-1 \leq t \leq 1$ and $s_{jm}$ is a spherical harmonic of degree $j$, 
	we may apply the Funk-Hecke formula~\cite[Theorem 6]{mueller} with $t = \inp{x}{z}$ (see also~\eqref{eq:FunkHecke}) to rewrite
	\begin{align}\label{eq:FunkHeckeApplied}
		\ints s_{j'm'}(z) (\inp{x}{z})^{2\ell} dz = \frac{\omega_{n-1}}{\omega_n} \lambda(n,\ell,j) s_{j'm'}(x)\ .
	\end{align}
	Combining~\eqref{eq:traceExpand} with~\eqref{eq:FunkHeckeApplied}, and using the orthogonality of the spherical harmonics then gives
	\begin{align}
		\tr\left(\sjm \sjmt\right) &= \frac{\omega_{n-1} \lambda(n,\ell,j)}{\omega_n} \ints s_{jm}(x) s_{j'm'}(x) dx\\
		&= \frac{\omega_{n-1} \lambda(n,\ell,j)}{\omega_n^2} \delta_{jj'} \delta_{mm'}\ ,
	\end{align}
	as promised.
\end{proof}

In addition to considering the trace of products of spherical harmonic matrices, we will also need to know the trace of any such matrix by itself. This is an easy consequence of 
the lemma above and properties of the spherical harmonics $s_{jm}(x)$.

\begin{lemma}\label{lem:sjmTraceless}
	Let $\ell, j, n \in \Natural$ and $n \geq 3$, then for all $j > 0$ the matrix $\sjm$ defined in~\eqref{eq:defineSJM} satisfies $\tr\left(\sjm\right) = 0$.
	Furthermore, $\tr\left(S_{0}^\ell\right) = 1/\sqrt{\omega_n}$.
\end{lemma}
\begin{proof}
	The first part follows immediately from the fact that $s_0(x)$ is the constant function and the orthogonality of the spherical harmonics
	\begin{align}\label{eq:genTrace}
		\tr\left(\sjm\right) = 
		\ints s_{jm}(x) dx \propto \ints s_{jm}(x) s_{0}(x) dx = 0\ .
	\end{align}
	The second part follows from the fact that $s_0(x) = c$ is given by a constant $c$, which is determined by the normalization condition of~\eqref{eq:sphericalNorm}
	\begin{align}
		\ints s_{0}(x)^2 dx = c^2 \ints dx = c^2 = \frac{1}{\omega_n}\ .
	\end{align}
	Hence, $s_0(x) = 1/\sqrt{\omega_n}$ which together with~\eqref{eq:genTrace} implies our claim.
\end{proof}

\section{Polynomials as matrices}

We follow the general approach of Parrilo~\cite{parrilo} and Laserre~\cite{lasserre} to express the class of polynomial optimization problems that we are interested in as a hierarchy of semidefinite programs (SDPs). 
To use this approach, we will need to transform the problem involving polynomials into an equivalent problem involving matrices. Positivity of the matrices is related to the existence of sum of squares decompositions for the polynomial.
The specific mapping that we use here appears to be somewhat different to the one that is usually used but we will see that it has significant advantages when it comes to analysing the convergence of the sequence of SDPs.

\subsection{Matrices for polynomials in complex variables}

As a warmup, we first discuss the case of Hermitian polynomial functions $T(\x,x^*) $ of \emph{complex} variables $x\in \Complex P^{n-1}$. This case already captures some of the essential ideas, but is slightly simpler than the case of real variables.

\subsubsection{Construction of matrices}
In order to map a polynomial to a matrix, consider a vector space $V=\Complex^n$ 
with orthonormal basis $\{\ket{j}, j \in {1,\ldots,n}\}$, where we identify each basis
element $\ket{j}$ with a particular variable $x_j$. Products of variables can then naturally be identified with the basis of $V^{\otimes \red}$ given by the tensor products
of the basis for $V$. An arbitrary a vector $\ket{x} \in V$ can thus be identified with the variables 
of the polynomial as follows
\begin{align}
		\ket{x} &:= \sum_{j=1}^n x_j \ket{j}\ .
\end{align}
For a Hermitian polynomial $T(x,x^*)$ of symmetric degree $d=2\red$ our goal will be to find a unique matrix $Z_T \in \End(V^{\otimes \red})$ such that for all $x$
\begin{align}\label{eq:poly1}
	(\bra{x}^{\otimes \red})Z_T(\ket{x}^{\otimes \red})= T(x,x^*)\ .
\end{align}
In what follows, we will also call $Z_T$ a \emph{polynomial matrix}, 
if $Z_T$ represents a polynomial $T(x,x^*)$ in the sense of~\eqref{eq:poly1}. A very strong motivation for this definition is that unitary transformations of the variables $|x\rangle \rightarrow U|x\rangle$ result in unitary transformations of the polynomial matrix $Z_T\rightarrow U^{\dagger \otimes \red}Z_TU^{\otimes \red}$ and this group action corresponds exactly to the original group action we defined on $T(x,x^*)$ such that $Z_{L_g[T]}=g^{\dagger \otimes \red}Z_Tg^{\otimes \red}$ for each $g\in \U(n)$. 

Let us now make a first attempt at finding $Z_T$.
One straightforward construction of a polynomial matrix is as follows
\begin{align}
\label{eq:PVdefn}
		Z_T' &:= 
	\sum_{
	|{\bf i} |= |{\bf j}| = \red} \alpha^{\bf i}_{\bf j} \ket{v_{\bf i}}	\bra{v_{\bf j}}
	\ ,
\end{align}
where $|v_{\bf i}\rangle := |1\rangle^{\otimes i_1}\otimes |2\rangle^{\otimes i_2}\ldots |n\rangle^{\otimes i_n}$. It is easily verified that~\eqref{eq:poly1} holds for $Z_T$ given above, which acts on the full space $V^{\otimes \red}$. The condition on $\alpha^{\bf i}_{\bf j} $ guarantees that $Z'_T$ is Hermitian and will therefore have real eigenvalues. We note that $Z'_T$ is sometimes also called the Gram matrix of $T$, and $Z_T$ the symmetric tensor according to $T$~\cite{burgdorf:comments}.

\begin{example}\label{ex:complexPoly}
To get some intuition about polynomial matrices, consider the following simple example. Let $T(x_1,x_2,x_3,x_4) = x_1^*x_2^* x_3 x_4 + 
x_3^* x_4^* x_1 x_2$. Note that $T$ is Hermitian and a sum of monomials of symmetric degree $d=4$. To see how we can write $T$ as a matrix consider $Z_T' \in \End(V^{\otimes 2})$ where $V = \Complex^4$
\begin{align}
	Z_T' = \outp{1,2}{3,4} + \outp{3,4}{1,2}\ .
\end{align}
Clearly $Z_T'$ is Hermitian and $(\bra{x}^{\otimes 2})Z_T'(\ket{x}^{\otimes 2})= T(x)$ as promised.
\end{example}

However, our first attempt at finding a mapping to matrices is somewhat unsatisfactory, both because the ordering of the factors of $V$ is arbitrary and because the polynomials are embedded into a vector space of dimension $n^{2\red}$ ($\dim(\End(V^{\otimes \red}))$) which is much larger than vector space of Hermitian polynomials which has dimension ${n+\red-1 \choose n-1}^2$. In the following we will define a polynomial matrix based on $Z_T$ that resolves these issues.
In particular, we will now see that the ``relevant'' part of $Z_T'$ can be chosen to be unique. 

Recall from Section~\ref{sec:repTheory} that $\ket{x}^{\otimes \red}$ has support only on the symmetric subspace. This implies that
\begin{align}
	\bra{x}^{\otimes \red}(\Pi_\red Z'_T \Pi_\red)\ket{x}^{\otimes \red} = \bra{x}^{\otimes \red}Z'_T \ket{x}^{\otimes \red}\ .
\end{align}
If we define the polynomial matrix 
\begin{equation}
\label{eq:Pdefn}
Z_T := \Pi_\red Z_T' \Pi_\red
\end{equation}
it still satisfies~\eqref{eq:poly1} and is clearly also Hermitian.
From~\eqref{eq:innerprod} it is easy to see that we can write $\Pi Z_T' \Pi$ in terms of the number state basis for the symmetric subspace as
\begin{align}
\label{eq:Pnumber}
	Z_T = 
	\sum_{
	|{\bf i} |= |{\bf j}| = \red}c_{\bf i}c_{\bf j} {\bf i!j!}\alpha^{\bf i}_{\bf j} \outp{{\bf i}}{{\bf j}} = \sum_{
	|{\bf i} |= |{\bf j}| = \red}\frac{\sqrt{ {\bf i!j!}}}{\red!}\alpha^{\bf i}_{\bf j} \outp{\bf i}{\bf j} \ .
\end{align}
Since the number state basis is orthonormal, it is clear that this mapping is now one-to-one from the Hermitian polynomials to the Hermitian matrices on the symmetric subspace of $V^{\otimes \red}$. 

\begin{example}Let us now revisit Example~\ref{ex:complexPoly} and show how $T$ can be expressed as a matrix $Z_T$ 
on the symmetric subspace. First of all, note that the element $x_1x_2$ corresponds to ${\bf i} = (1,1,0,0)$ and $x_3x_4$ to 
${\bf i}  = (0,0,1,1)$. In both cases ${\bf i}! = 1! 1!$. Hence from~\eqref{eq:Pnumber}
we have 
\begin{align}
	Z_T =\frac{1}{4!}\left(\outp{(1,1,0,0)}{(0,0,1,1)} + \outp{(0,0,1,1)}{(1,1,0,0)}\right)\ .
\end{align}
\end{example}

Note that this mapping is different from one that is often used in the literature~\cite{parrilo} in which $T$ is mapped to an operator on a vector space with its basis vectors identified directly with monomials, in our notation 
\begin{align}
	\tilde{Z}_T = 
	\sum_{
	|{\bf i} |= |{\bf j}| = \red}\alpha^{\bf i}_{\bf j} \outp{\bf i}{\bf j}\ .
\end{align} 
Note that this definition does not meet our requirement~\eqref{eq:poly1} which will prove to be essential to our analysis in terms of a de 
Finetti theorem.

\subsubsection{Operations on polynomials}
It is instructive to see that many natural operations on polynomials have an analogue in terms of matrix operations for our definition of polynomial matrices.

We first consider the extension of the polynomial to a polynomial of higher degree. More specifically, 
consider the operation $T(x)\rightarrow T'(x)=T(x)r^2(x)$ which maps an order $d=2\red$ polynomial to order $d+2$. If we define the unit vector ${\bf e}_k$ in the obvious way then from~\eqref{eq:polydefnc} we have 
\begin{align}
\label{eq:polydefn}
	T'({\x}) = \sum_{
	|{\bf i} |= |{\bf j}| = \red} \sum_k
	\alpha^{\bf i}_{\bf j}{\x^{*i+e_k}\x^{j+e_k}}\ ,
\end{align}
and so we find 
\begin{align}
	Z_{T'}
	= \sum_{
	|{\bf i} |= |{\bf j}| = \red}\sum_k\frac{\sqrt{ {\bf (i+e_k)!(j+e_k)!}}}{d!}\alpha^{\bf i}_{\bf j} \outp{{\bf i+e_k}}{{\bf j+e_k}}=\Pi (Z_T\otimes I_n)\Pi \ .
\end{align}
The first expression comes from~\eqref{eq:Pnumber}, while the second is clear from the observation that $\langle x|I_n|x\rangle=\sum_k x_k^*x_k$ and the definition of $Z_T'$~\eqref{eq:PVdefn} and $Z_T$~\eqref{eq:Pdefn}.  

Clearly this formula extends such that if $T'(x)=T(x)r^{2(\ell-a)}(x)$ then we have $Z_{T'}=\Pi_\ell (Z_T\otimes I_n^{\otimes (\ell-a)})\Pi_\ell$. A useful identity is the following, consider a matrix $M_\ell \in  \End(\Sym(V^{\otimes \ell}))$ and its partial trace $M_a=\tr_{\downarrow a}(M_\ell)$ then we have
\begin{equation}
\label{eq:complexpartial}
\tr(M_\ell Z_{T'})=\tr(M_\ell (Z_T\otimes I_n^{\otimes (\ell-a)}))=\tr(M_aZ_T).
\end{equation}
The first equality comes because $\Pi_\ell M_\ell\Pi_\ell=M_\ell$ and the cyclic property of the trace, while the second is a result of the definition of partial trace.

Another natural operation we can perform on a polynomial is to act with the Laplacian
which maps a Hermitian polynomial of degree $2\red$ to a Hermitian polynomial of degree $2\red-2$. We find
\begin{align}
\label{eq:lagrangian}
\Delta T({\x}) &= \sum_{
	|{\bf i} |= |{\bf j}| = a} \sum_k 
	i_kj_k\alpha^{\bf i}_{\bf j}{\x^{* {\bf i-e_k}}\x^{{\bf j-e_k}}}\\
&= \sum_{
	|{\bf i} |= |{\bf j}| = a-1}\left( \sum_k 
	(i_k+1)(j_k+1)\alpha^{\bf i+e_k}_{\bf j+e_k}\right){\x^{*{\bf i}}\x^{{\bf j}}}\ .
\end{align}
On the other hand we can take the partial trace of $T$ using \eqref{eq:ptrace} to find
\begin{align}
	{\tr}_1 (Z_T) 
	&= \frac{1}{a}\sum_t \sum_{
	|{\bf i} |= |{\bf j}| = a}\sqrt{i_tj_t}\frac{\sqrt{ {\bf i!j!}}}{a!}\alpha^{\bf i}_{\bf j} \outp{\bf i-e_t}{\bf j-e_t} \\
	&= \frac{1}{a^2} \sum_{
	|{\bf i} |= |{\bf j}| = a-1}\left(\sum_t (i_t+1)(j_t+1)\alpha^{\bf i+e_t}_{\bf j+e_t} \right)\frac{\sqrt{ {\bf i!j!}}}{(a-1)!}\outp{\bf i}{\bf j} \ .
\end{align}
and by comparison with~\eqref{eq:Pnumber} we see that $a^2{\tr}_1(Z_T) $ is the matrix that represents the polynomial $\Delta T(x)$, and so up to a scale factor we can interpret the partial trace as the Laplacian.

\subsubsection{Sum of squares decompositions}
The main purpose of mapping polynomials to matrices in this fashion is to establish the existence or otherwise of a sum of squares decomposition in which a polynomial of symmetric degree $d$ is written as a sum over squares of degree $d/2$ polynomials.

Just as we established a natural mapping of monomials of symmetric degree $d$ to matrices on $\End(\Sym(V^{\otimes d/2}))$  it is possible to establish a mapping of monomials $x^{\bf i}$ onto vectors in $\Sym(V^{\otimes d/2})$. Our goal will be to find a \emph{unique} mapping from complex polynomials of degree $(0,a)$
\begin{align}
\label{eq:polydefnsos}
	T({\x}) = \sum_{
	|{\bf i} |= a} 
	\alpha_{{\bf i}}{\x^{i}}\ ,
\end{align}
to a vector, where we want that
\begin{align}
\label{eq:polycorrespsos}
T(x) = \langle Z_T\ket{x}^{\otimes a} \ .
\end{align}
Proceeding as before we first
consider the following trial vector
\begin{align}
\label{eq:pVdefnsos}
	\ket{Z_{T}'} &:= 
	\sum_{
	|{\bf i} |=a} \alpha_{\bf i}	\ket{v_{\bf i}}
	\ ,
\end{align}
where as before $|v_{\bf i}\rangle = |1\rangle^{\otimes i_1}\otimes |2\rangle^{\otimes i_2}\ldots |n\rangle^{\otimes i_n}$. 
Using the same argument as before $|x\rangle^{\otimes a}$ lives on the symmetric subspace $\Sym(V^{\otimes a})$ and we project \emph{the vector} $\bra{Z_{T}'}$ onto 
that subspace to define
\begin{align}
\label{eq:pdefnsos}
		\langle{Z_T}| &:= \langle{Z'_T}| \Pi_a=
	\sum_{
	|{\bf i} |=a} \alpha_{\bf i}	c_{\bf i}{\bf i!} \bra{{\bf i}} = \sum_{
	|{\bf i} |=a} \sqrt{\frac{{\bf i!}}{a!}}\alpha_{\bf i}\bra{{\bf i}}
	\ ,
\end{align}
where the second expression uses $\langle {\bf i}|v_{\bf k}\rangle=c_{\bf i}{\bf i!}\delta_{{\bf i,k}}$. Clearly we can also associate the complex conjugated polynomial $T^*(x,x^*)$ with the vector $|Z_T\rangle$.
Again the orthogonality of the number state basis implies that this mapping is unique. As a result we may map any vector $|Z\rangle\in \Sym(V^{\otimes a})$ uniquely to a polynomial in $T_Z$ of the form~\eqref{eq:polydefnsos}. 

Now suppose that polynomial $T(x,x^*)$ is such that $T'(x,x^*)=T(x,x^*)r^{2k}$ maps to $Z_{T'}\geq 0$. We may write
\begin{equation}
Z_{T'} = \sum_i \lambda_i |Z_i\rangle \langle Z_i|  
\end{equation}
where each eigenvalue $\lambda_i\geq 0$ and the eigenvectors $|Z_i\rangle \in \Sym(V^{\otimes ( a+k)})$. As a result there exist $T_i(x)$ such that
\begin{equation}
T'(x,x^*)=T(x,x^*)r^{2k} = \sum_i\lambda_i|T_i(x)|^2
\end{equation}
the right hand side is known as a {\it sum of squares} decomposition for $T'(x,x^*)$. The existence of this sum of square decomposition implies that $T(x,x^*)\geq 0$ for all $x$. The set of polynomials $T(x,x^*)$ of symmetric order $d=2\red$ such that $T'(x,x^*)=T(x,x^*)r^{2(\ell-a)}$ has a sum of square decomposition will be labelled $\mathcal{S}_C(d,\ell)$.

\subsection{Matrices for polynomials in real variables}
\label{sec:realpoly}

Let us now consider the case of real variables. Our goal as before will be to find a \emph{unique} mapping from real polynomials $T(x)$ of even degree $d=2\red$
to matrices $Z_T$ such that $\bra{x}^{\otimes a}Z_T \ket{x}^{\otimes a} = T(x)$~\footnote{Recall from Section~\ref{sec:evenoddDegree} 
that the optimization of any homogeneous polynomial can be
performed by optimization of a homogeneous polynomial of even degree.}. Naturally, we could proceed just as in the complex case. However, whereas our mapping was indeed unique for the case of Hermitian polynomials, it no longer is for the case of real valued variables. To see this, is it useful to think of the matrix $Z_T$ as a vector, where we want that
\begin{align}
\label{eq:polycorresp}
T(x) = \bra{x}^{\otimes d} \ket{Z_T}\ .
\end{align}
Note that when proceeding as in the complex case above $\ket{Z_T}$ is of the form
\begin{align}
\ket{Z_T} = \sum_{{\bf i}{\bf j}} \beta^{\bf i}_{\bf j} \underbrace{\ket{{\bf i}}}_{x_{i_1}^* \ldots x_{i_n}^*}
\underbrace{\ket{{\bf j}}}_{x_{i_1}\ldots x_{i_n}}\ ,
\end{align}
for some coefficients $\beta^{\bf i}_{\bf j}$. That is, the first $V^{\otimes a}$ systems correspond to the conjugate of variables and the second $V^{\otimes a}$ to non-conjugated ones.
Clearly, in the real case there is no such separation between conjugated and non-conjugated variables and there are many ways to subdivide them into two components to form a matrix. 
Hence, the mapping is no longer unique.

\subsubsection{Construction of matrices}
However, it becomes clear when looking at~\eqref{eq:polycorresp} that our problem now has an additional symmetry that again allows us to restrict to a certain subspace on which polynomials are represented uniquely as matrices. In analogy to our discussion of complex polynomials of order $(0,a)$ in the previous section we can uniquely map our real polynomials to vectors in $\Sym(V^{\otimes d})$ as follows  \begin{align}
\label{eq:pdefnsosConstruct}
\ket{Z_T} &:= 
	 \sum_{
	|{\bf i} |=d} \sqrt{\frac{{\bf i!}}{d!}}\alpha_{\bf i}\ket{{\bf i}}
	\ .
\end{align} 
We may map this uniquely to a matrix acting on $\Sym(V^{\otimes d})$ in the usual way but the corresponding matrix $Z_T$ lives in the \emph{maximally} symmetric subspace, i.e., 
$Z_T \in \maxsymR$, as opposed to the merely symmetric subspace in the complex case. As we noted earlier all matrices in $\maxsymR$ are symmetric and therefore have 
real eigenvalues. This fact is important later for sum of squares decompositions. In analogy to the complex case orthogonal transformations of the variables $|x\rangle \rightarrow O|x\rangle$ result in orthogonal transformations of the polynomial matrix $Z_T\rightarrow O^{T \otimes \red}Z_TO^{\otimes \red}$ and this group action corresponds exactly to the original group action we defined on $T(x)$ such that $Z_{L_g[T]}=g^{T \otimes \red}Z_Tg^{\otimes \red}$ for each $g\in O(n)$.


\subsubsection{Operations on polynomials}
It is again instructive to consider how natural operations on polynomials map to operations on matrices, or their vectorized forms. 

Firstly we will define an operation on the vector $\ket{Z_T}$ corresponding to $T(x)$ to produce the matrix corresponding to $T'(x)=T(x)r^2(x)$. In order to do this we define the vector $|r^2\rangle = \sum\ket{ii}$ which corresponds to $r^2(x)=\langle x |^{\otimes 2}\ket{r^2}$. Then following the mappings discussed above we define $\ket{Z'_{T'}}=\ket{Z_T}\otimes \ket{r^2}$ for which it is clear that $\langle x|^{\otimes (d+2)}\ket{Z'_{T'}}$ and so as before a projection onto the symmetric subspace preserves this property and we define $\ket{Z_{T'}}=\Pi_{d}(\ket{Z_T}\otimes \ket{\Psi})$. Notice that the matrix corresponding to $\ket{r^2}$ is $\id_n$ and so we could also use the notation $|\id_d\rangle$ instead. However in contrast to the complex case $Z_{T'}\neq \Pi_a (Z_T\otimes \id_n)\Pi_a$ as is clear since the resulting matrix is not maximally symmetric.

As in the complex case this formula extends such that if $T'(x)=T(x)r^{2(\ell-a)}(x)$ then we have $\ket{Z_{T'}}=\Pi_{2\ell} (\ket{Z_T}\otimes \ket{\id_n}^{\otimes (\ell-a)})$. A useful identity is the following, consider a matrix $M_\ell \in  \maxsymEll$ and its partial trace $M_a=\tr_{\downarrow a}(M_\ell)$ then we have
\begin{eqnarray}
\label{eq:realpartial}
\tr(M_\ell Z_{T'})&=&\langle M_\ell | Z_{T'}\rangle= \bra{M_\ell}(\ket{Z_T}\otimes \ket{\id_n}^{\otimes (\ell-a)})=\tr(M_\ell (Z_T\otimes \id_n^{\otimes (\ell-a)})) \nonumber \\ &=&\tr(M_aZ_T).
\end{eqnarray}
The first equality is the definition of these matrices in terms of vectors. The second identity holds since $\ket{M_\ell} \in \Sym(V^{\otimes 2\ell})$. The third equality converts the vector inner product back into an inner product of the corresponding matrices, and the final equality is the definition of the partial trace. 

We can also develop operations that correspond to the Laplacian acting on the $Z_T$. In order to do this, let us define the following operation from $V^{\otimes d}$ to $V^{\otimes d-2}$
\begin{equation}
\theta_{(1,2)}\ket{x^{(1)}}\ket{x^{(2)}}\ket{x^{(3)}}\ldots \ket{x^{(d)}}= \inp{x^{(1)}}{x^{(2)}} \ket{x^{(3)}}\ldots \ket{x^{(d)}}.
\end{equation}
We can then use permutations to define $\theta_{(i,j)}$ for all $(i,j)$ in the obvious way and then define the following operation which commutes with all permutations
\begin{equation}
\theta_1=\sum_{i,j\leq i} \theta_{(i,j)}\ .
\end{equation}
It should be noted that this operation has a connection with representations of the orthogonal group~\cite{weyl:book}, but we will not
need this here.  The effect of $\theta_1$ on our basis for $V^{\otimes d}$ is straightforward to compute
\begin{equation}
\theta_1 |v_{\bf i}\rangle = \sum_k i_k(i_k-1) |v_{{\bf i}-2{\bf e}_k}\rangle \ .
\end{equation}
As a result we find the following
\begin{equation}
\theta_1 |Z_T' \rangle =\sum_{|{\bf i}|=d}  \sum_k i_k(i_k-1)\alpha_{{\bf i}}|v_{{\bf i}-2{\bf e}_k}\rangle
\end{equation}
and this is clearly the vector $|\Delta Z_T\rangle$ corresponding to $\Delta T(x)$ on $V^{\otimes d}$. Since $\theta_1$ and $\Pi_{d}$ commute we find that
\begin{equation}\label{eq:deltaTheta}
\theta_1|Z_T\rangle = |\Delta Z_T\rangle \ .
\end{equation}

We can now understand the partial trace on $Z_T$ as follows. For any matrix $M \in \End( (\Complex^n)^{\otimes \red})$
we have
\begin{align}
	\tr_1(M) = \sum_{i=1}^n (\id_{\red-1} \otimes \bra{i}) M (\id_{\red-1} \otimes \ket{i})\ .
\end{align}
Associating $M$ with a vector $\ket{M} = \sum_{ij} M_{ij} \ket{i}\ket{j} \in (\Complex^n)^{\otimes 2\red}$ we see that the partial trace is equivalent to computing
\begin{align}
	\tr_1(M) = \sum_{i=1}^{n} (\id \otimes \bra{i})(\id \otimes \bra{i}) \ket{M} = \theta_{(\red,2\red)}\ket{M}\ .
\end{align}
Note that when applied to a number state which is permutation invariant
\begin{align}
	d(d-1)\theta_{(\red,2\red)}\ket{\textbf{i}} = \theta_1 \ket{\textbf{i}}\ .
\end{align}
Hence, for any $\ket{M}$ of the form $\ket{M} = \sum_{\textbf{i}} \alpha_{\textbf{i}} \ket{\textbf{i}}$ we have $d(d-1)\tr_1(M) =  \theta_1 \ket{M}$ and thus
\begin{align}
	 Z_{\Delta T}=d(d-1)\tr_1(Z_T) \ .
\end{align}
so that partial trace once again corresponds to the Laplacian. 

\subsubsection{Sum of squares decomposition}
\label{sec:sosreal}

In this section we wish to develop a semidefinite programming method to find sum of squares decompositions for $T(x)$.
In the previous section we established a unique mapping of a real homogenous polynomial of order $d$ to a unique matrix $Z_T\in \maxsymR$. However $\maxsymR$ is not closed under multiplication and so some eigenvalues of $Z_T$ could be negative. Yet, there could be some symmetric (but not maximally symmetric) matrix $\bar{Z}_T\in\End(\Sym(V^{\otimes a}))$ such that $\langle x|^{\otimes a}\bar{Z}_T|x\rangle^{\otimes a}=T(x)$ and $\bar{Z}_T\geq 0$. This would imply the existence of a sum of squares decomposition for $T$ as we will see below. The required property is that $\bar{Z}_T=Z_T+\bar{Z}$ for some $|\bar{Z}\rangle$ such that $\Pi_{d}\ket{\bar{Z}}=0$ so that $\langle x|^{\otimes d}|\bar{Z}\rangle =0$ and we have $\langle x|^{\otimes d}|\bar{Z}_T\rangle=T(x)$.  

Not all positive polynomials $T(x)$ can be written as sums of squares but as in the complex case we can find procedures to check whether or not $T'(x)=T(x)r^{2\ell-d}$ can be written as a sum of squares.  We will say that $T(x)\in \mathcal{S}(d,\ell)$ when $T$ is a real polynomial of degree $d$ if there exists some $\bar{Z}\in \End(\Sym(V^{\otimes \ell}))$ having $\Pi_{2\ell}\ket{\bar{Z}}=0$ such that
\begin{equation}
Z_{T'}+\bar{Z}\geq 0.
\end{equation}
If such a $\bar{Z}$ exists we have 
\begin{equation}
	Z_{T'}+\bar{Z} = \sum_i \lambda_i |Z_{T_i}\rangle \langle Z_{T_i}|  
\end{equation}
where each eigenvalue $\lambda_i\geq 0$ and the eigenvectors $|Z_{T_i}\rangle \in \Sym(V^{\otimes \ell})$. As a result there exist $T_i(x)$ such that
\begin{equation}
T'(x)=T(x)r^{2\ell-d}(x) = \sum_i\lambda_i(T_i(x))^2
\end{equation}
the right hand side is known as a {\it sum of squares} decomposition for $T'(x)$. The existence of this sum of square decomposition implies that $T(x)\geq 0$ for all $x$.

\subsubsection{Decomposing polynomial matrices using spherical harmonics}

Above, we determined that our polynomial matrices can be represented on the maximally symmetric subspace $\maxsymR$. 
As such, it is not hard to see that there exists another way to find the corresponding matrices explicitly, if we already know
the decomposition of the polynomial into its Fourier components 
\begin{align}
T(x) = \sum_{j=0}^{2a} \sum_{m=1}^{N(j,n)} t_{jm} s_{jm}(x)\ .
\end{align}
Let us first consider how to express $s_{jm}(x)$ itself as a matrix $M_{s_{jm}}$. 
Note that by the Funk-Hecke formula (see also Lemma~\ref{lem:sjmTrace}), we have 
\begin{align}
\bra{x}^{\otimes a} \sjmR \ket{x}^{\otimes a} = \ints (\inp{x}{z})^{2a} s_{jm}(z) dz = \frac{\omega_{n-1}}{\omega_n} \lambda(n,a,j) s_{jm}(x)\ ,
\end{align}
where $\omega_n$ is the surface area of the sphere $\sphere$ and $\lambda(n,a,j)$ is a normalizaton defined in Lemma~\ref{lem:simplifyLambda}. 
Letting 
\begin{align}
M_{s_{jm}} := \frac{\omega_n}{\omega_{n-1}\lambda(n,a,j)} \sjmR
\end{align}
is thus the corresponding maximally symmetric matrix. By linearity we can thus write the polynomial matrix as
\begin{align}
Z_T = \sum_{jm} t_{jm} M_{s_{jm}}\ .
\end{align}

\section{Moments and moment matrices}\label{sec:moments}

\subsection{Moment matrices for complex variables}

Once again, we consider the case of complex variables first, where much is known from quantum optics and quantum information in which complex vector spaces are the most relevant ones. 


Let us briefly recall the concepts of moment sequences and moment matrices in optimization 
(see~\cite{monique:survey} for an excellent survey) and translate them to our notation and symmetrized form.
Let $\mu$ be a measure on $\Complex P^{n-1}$.
Given a monomial $x^{*\bf i} x^{\bf j}$ with $|{\bf i}|=|{\bf j}|=a$ as described above we can define a moment of symmetric order $2a$ as follows 
\begin{align}\label{eq:mu}
y_{\bf i, \bf j} = \int \hat{x}^{* \bf i}\hat{x}^{\bf j} d\mu(x)\ .
\end{align} 
We define the moments in terms of components of the normalised vector $\hat{x}=x/\|x\|_2$ so that the integrand is a proper function on $\Complex P^{n-1}$.
A question of great concern is to determine whether some arbitrary (truncated) sequences $y = \{y_{\bf i,\bf j}\}_{\bf i,\bf j}$ can
be written in the form of~\eqref{eq:mu} for some measure $d\mu$. Such a measure is also known as a \emph{representing measure}.
One of the objectives of Lasserre's hierarchy is to provide us with a procedure that we can use to determine whether $y$ admits a representing measure and in general this problem can be partially addressed by forming suitable matrices of moments and testing for positivity of the matrix.

Let $Z_{\bf i,\bf j}$ denote the polynomial matrix for $x^{*\bf i} x^{\bf j}$ as defined in the previous section. Then the corresponding moment can be recovered from this matrix as
\begin{align}
\int_{\Complex P^{n-1}} \bra{\hat{x}}^{\otimes a} Z_{\bf i,\bf j} \ket{\hat{x}}^{\otimes a} d\mu(x) = 
\tr\left[\underbrace{\left(\int \proj{\hat{x}}^{\otimes a} d\mu(x)\right)}_{=:M_a(\mu)} Z_{\bf i, \bf j}\right] = y_{\bf i,\bf j}\ .
\end{align}
Here we see that the matrix $M_a(\mu)\in \End(\Sym(V^{\otimes a}))$ contains information about all the moments of symmetric order $2a$ and is known as a {\it truncated moment matrix}~\cite{monique:survey}. The particular arrangement of moments in this matrix differs from that of many truncated moment matrices that can be found in the literature but it is chosen to correspond to our chosen mapping of polynomials to matrices and has many useful properties. For example, as always in such constructions, $M_a(\mu)\geq 0$. Also normalization of the measure requires that $\tr(M_a(\mu))=1$. For all $a'\leq a$ we have $M_{a'}(\mu)=\tr_{\downarrow a'}(M_a(\mu))$. Finally as a result of this definition we can evaluate the average of a polynomial of symmetric degree $2a$ over the measure $d\mu(x)$ as follows
\begin{align}
\tr(M_a(\mu) Z_T) 
&= \int_{\Complex P^{n-1}}  \bra{\hat{x}}^{\otimes a} Z_T \ket{\hat{x}}^{\otimes a} d\mu(x) \nonumber \\
&= \int_{\Complex P^{n-1}} T(\hat{x}) d\mu(x)\ .
\end{align}
This justifies our choice of definition of the truncated moment matrices $M_a$. This average value can also be inferred from the truncated moment matrix $M_\ell(\mu)$ for all $\ell>a$. Let us define $T'(x)=T(x)r(x)^{2(\ell-a)}$, and the corresponding matrix $Z_{T'}$. We then have by~\eqref{eq:complexpartial} 
\begin{equation}
\tr(M_\ell(\mu)Z_{T'})=\tr(M_a(\mu)Z_T)=  \int_{\Complex P^{n-1}} T(\hat{x}) d\mu(x)\ .
\end{equation}
The set of valid truncated moment matrices can be defined as follows
\begin{align}
	\mathcal{M}_C(a)=\{M\in \End(\Sym(V^{\otimes a})) | \exists  \ d\mu(x):M=\int_{\Complex P^{n-1}}\proj{\hat{x}}^{\otimes a} d\mu(x)\}\ . 
\end{align}
Testing whether a representing measure exists for a matrix $M\in \End(\Sym(V^{\otimes a})) $ so that $M\in \mathcal{M}_C(a)$ is a difficult task in general. That is,
it is difficult to decide membership of $\mathcal{M}_{C}(a)$.
However, there 
are some straightforward necessary conditions for membership of $\mathcal{M}_C(a)$ which we can test. As noted above we must have $M\geq 0$ and $\tr(M)=1$, i.e., $M$ is a state.
Furthermore, $M$ is extendible in the sense that for all $\ell>a$ there exists a matrix $M_\ell\in  \End(\Sym(V^{\otimes \ell})) $ such that $M_\ell \geq 0$, $\tr(M_\ell) = 1$ 
and $\tr_{\ell-a}(M_\ell) = M$. For a given $M$ the question of whether such an $M_\ell$ exists is an instance of a semidefinite programming feasibility problem and is computationally tractable~\cite{boyd:book}. This leads us to define the following set 
\begin{align}
	\mathcal{M}_C(a,\ell)=&\left\{M\in \End(\Sym(V^{\otimes a}))  |  \exists  M_\ell \in  \End(\Sym(V^{\otimes \ell}) : M_\ell \geq 0,\right. \nonumber \\ &  \quad \left. \tr(M_\ell)=1,\ \tr_{\ell-a}(M_\ell) = M\right\} 
\end{align}
and we have the containment $\mathcal{M}_C(a)\subset \mathcal{M}_C(a,\ell) \  \forall \ell$.

\subsection{Moment matrices for real variables}
\label{sec:momentsreal}

Let us now turn to the case of real variables. As it turns out the main ideas are essentially analogous.
In analogy to the above, one defines a 
\emph{moment of order ${\bf i}$} with
as $y_{\bf i} = \int x^{\bf i} d\mu(x)$, where $x^{\bf i}$ is a monomial determined by ${\bf i}$ as outlined above~\cite{monique:survey} and $d\mu(x)$ is a normalised measure over the hypersphere $S^{n-1}$. As before we are interested in determining whether a given sequence of numbers $\{y_{\bf i}\}_{\bf i}$ can be true moments for some representing measure. Once again necessary conditions for this can be determined that are related to matrix positivity.

Let $Z_{\bf i}$ denote the (infinite) polynomial matrix for $x^{\bf i}$ as defined in the previous section. Note that $Z_{\bf i} \in \maxsymR$, 
that is, $Z_{\bf i}$ is a maximally symmetric matrix.
We can now recover the moment $y_{\bf i}$ 
using this matrix as
\begin{align}
\int \bra{x}^{\otimes 2a} \ket{Z_{\bf i}} d\mu(x) = 
\tr\left[\underbrace{\left(\ints \proj{x}^{\otimes a} d\mu(x)\right)}_{=:M_a(\mu)} Z_{\bf i}\right] = y_{\bf i}\ .
\end{align}
Here we see that the matrix $M_a(\mu)$ contains information about all the moments of order $d=2a$ and is known as a {\it truncated moment matrix}~\cite{monique:survey}. As in the complex case we have $M_a(\mu)\geq 0$, $\tr(M)=1$ due to normalization of the measure and for all $a'\leq a$ we have $M_{a'}(\mu)=\tr_{\downarrow a'}(M_a(\mu))$. Notice, however, that since now $x\in S^{n-1}$ we have $\proj{x}^{\otimes a} \in \maxsymR$ and thus also $M_a(\mu)\in \maxsymR$. Finally as a result of this definition we can evaluate the average of a polynomial of degree $2a$ over the measure $d\mu(x)$ as in the complex case
\begin{align}
\tr(M_a(\mu) Z_T) 
&= \int_{S^{n-1}}  \bra{x}^{\otimes a} Z_T \ket{x}^{\otimes a} d\mu(x) \nonumber \\
&= \int_{S^{n-1}} T(x) d\mu(x)\ . \label{eq:polyav}
\end{align}
As before this average value can be inferred from the truncated moment matrix $M_\ell(\mu)$ for all $\ell>a$. As before, let us define $T'(x)=T(x)r(x)^{2(\ell-a)}$, and 
let $Z_{T'}$ denote its polynomial matrix. We then have by~\eqref{eq:realpartial} 
\begin{equation}
\tr(M_\ell(\mu)Z_{T'})=\tr(M_a(\mu)Z_T)=   \int_{S^{n-1}} T(x) d\mu(x)\ .
\end{equation}
The set of valid truncated moment matrices can be defined as follows
\begin{align}
	\mathcal{M}_{+}(a)=\{M\in \maxsymR  | \exists  \ d\mu(x):M=\int_{\Complex P^{n-1}}\proj{x}^{\otimes a} d\mu(x)\}\ . 
\end{align}
Testing whether a representing measure exists for a matrix $M\in \maxsymR $ so that $M\in \mathcal{M}_{+}(a)$ remains a difficult task. However, as in the complex case there 
are several straightforward necessary conditions for membership of $\mathcal{M}_{+}(a)$. As noted above we must have $M\geq 0$ and $\tr(M)=1$. We also again want that $M$ is extendible
such that for any $\ell>a$ there exists a matrix $M_\ell\in  \maxsymEll $ satisfying $M_\ell \geq 0$, $\tr(M_\ell) = 1$ and $\tr_{\ell-a}(M_\ell) = M$. For a given $M$ the question of whether such an $M_\ell$ exists is once again an instance of a semidefinite programming feasibility problem. That is, membership of the following set
\begin{align}
	\mathcal{M}(a,\ell)=&\left\{M\in \maxsymR |  \exists  M_\ell \in  \maxsymEll : M_\ell \geq 0,\right. \nonumber \\ &  \quad \left. \tr(M_\ell)=1,\ \tr_{\ell-a}(M_\ell) = M\right\} 
\end{align}
can be decided using a semidefinite program. Clearly, we again have the containment $\mathcal{M}_{+}(a)\subset \mathcal{M}(a,\ell) \  \forall \ell$.

\section{$P$ and $Q$-representations}\label{sec:PvsQ}

In our proof of convergence, we make use of two well known concepts originating in quantum optics~\cite{scully}. 
These are the so-called Glauber-Sudarshan $P$-representation and Husimi $Q$-representation of a matrix. We will see that 
they have a close connection to what are known as moment matrices in optimization in our symmetrized form of the SDP hierachy.

\subsection{On complex projective spaces}

\subsubsection{$P$-representation}
Let us now consider an arbitrary matrix $M \in \End(\Sym((\Complex^n)^{\otimes \ell}))$ with $M \geq 0$. It is well-known from quantum 
optics (see e.g.~\cite{scully} or~\cite{matthias:tomo} for the form used here) that any such $M$ can be written as
\begin{align}\label{eq:NPrep}
	M = \int_{\Complex P^{n-1}} P_M(x) \proj{\hat{x}}^{\otimes \ell} dx\ ,
\end{align}
for some coefficients given by the function $P_M \in L^2(\Complex P^{n-1})$, i.e. the space of square-integrable functions on the complex projective space $\Complex P^{n-1}$ (see~\cite{matthias:tomo} for a quantum information treatment). 
This~\eqref{eq:NPrep} is known as the \emph{$P$-representation} of $M$. 
If there exists a $P_M$ such that for all $x$ we have $P_M(x) \geq 0$ we also say that $M$ admits a \emph{positive $P$-representation}. 
Note that for $\tr(M) = 1$
\begin{align}
1 = \tr(M) = 
\int_{\Complex P^{n-1}} P_M(x) (\inp{\hat{x}}{\hat{x}})^\ell dx = 
\int_{\Complex P^{n-1}} P_M(x) dx \ .
\end{align}
As such, we may think of the numbers $P_M(x)dx$ as a quasi-probability distribution/measure, and indeed one exists for \emph{any} such $M$. 
The question of whether the sequence of moments corresponding to $M$ admits a representing measure is thus equivalent to asking 
whether $M$ admits a \emph{positive} $P$-representation. As such, our symmetrized way of writing polynomials makes it clear
that there is a close relation between the question asked in optimization of whether a representing measure can be found, and the question that is of some concern in quantum optics of whether $M$ admits a positive $P$-representation.  


\subsubsection{$Q$-representation}
Closely related to the $P$-representation is the so-called $Q$-representation of $M \in \Sym((\Complex^n)^{\otimes \ell})$, 
defined by the value of
\begin{align}
Q_M(x) := \bra{x}^{\otimes \ell} M\ket{x}^{\otimes \ell} 
\end{align}
for all $\ket{x} \in \Complex^n$. 
In the language of the previous section, we are now simply inverting the mapping of Hermitian polynomials to Hermitian matrices and thinking of $M$ as a polynomial matrix
for the polynomial $Q_M$. As we know this mapping is one-to-one and invertible and so $M$ is determined completely by the values 
of $Q_M$. It is also useful to note that for $M \geq 0$, we have $Q_M(x) \geq 0$ for all $x$ so that while a normalised $P$-function even for a positive $M$ does not have an interpretation as a measure we may construct valid measures using $Q$-functions and positive matrices. 
For the complex case, it is known that there exists a way to convert a matrix' $P$-representation into its $Q$-representation and vice
versa for so-called coherent states~\cite{scully} or finite dimensional systems~\cite{matthias:tomo}. 

\subsubsection{Evaluation on a polynomial matrix}\label{sec:complexEval}

Before turning to the real case, it will be instructive to understand $\tr(M Z_T)$ where $M \geq 0$ with $M \in \Sym((\Complex^n)^{\otimes \ell})$ and $Z_T$ is the polynomial matrix of some polynomial $T(x)$. 
Because such an $M$ always admits a $P$-representation we have
\begin{align}
\tr(M Z_T) 
&= \int_{\Complex P^{n-1}} P_M(x) \bra{\hat{x}}^{\otimes \ell} Z_T \ket{\hat{x}}^{\otimes \ell} dx \nonumber \\
&= \int_{\Complex P^{n-1}} P_M(x) T(\hat{x}) dx \label{eq:ppolyevalcomplex}\ .
\end{align}
That is, we can always understand this trace as the evaluation of the polynomial $T$ at points $x$ weighted by coefficients
$P_M(x)$. As these are not necessarily positive, $P_M(x) dx$ is not necessarily a measure. This observation will lead to a nice interpretation of solutions of intermediate (non-optimal) levels of the SDP hierarchy discussed later.
However, for any $M$ with $\tr(M)=1$ that admits a positive $P$-representation 
we do of course have such a representing measure which will be the case if the final (optimal) level of the
SDP hierarchy is reached.

\subsection{On the hypersphere}

Let us now turn to the case of real variables. As it turns out the main ideas are essentially analogous, however, we need to prove the existence of a similar $P$-representation. For any $M\in \maxsymEll$ we would like there to exist a function $P_M(x)$ on $S^{n-1}$ such that
\begin{align}\label{eq:realP}
	M = \int_{S^{n-1}} P_M(x) \proj{x}^{\otimes \ell} dx\ .
\end{align}
We will use standard arguments to establish this. We start by noting that all $M\in \maxsymEll$ do possess $Q$-representations.

\subsubsection{$Q$-representation}
Just like in the complex case, we can again consider a $Q$-representation for any $M \in \maxsymEll$, 
defined by the value of
\begin{align}\label{eq:qfunc}
Q_M(x) := \bra{x}^{\otimes \ell} M\ket{x}^{\otimes \ell} 
\end{align}
for all $x \in \sphere$.
As before, we are now again simply thinking of $M$ itself as a polynomial matrix
for the polynomial $Q_M$ and the mapping we described in the previous section was one-to-one and invertible, only zero matrices map to zero polynomials and vice versa. Note that any maximally symmetric matrix $M \in \maxsymEll$ corresponds to a polynomial of degree 
at most $2\ell$ and hence $M$ is fully determined by~\eqref{eq:qfunc} evaluated at $2\ell + 1$ distinct points. 
Clearly, for $M \geq 0$, we have $Q_M(x) \geq 0$ for all $x$. 

\subsubsection{Existence of a $P$-representation}\label{sec:PrealRep}
The following lemma shows that any maximally symmetric matrix does indeed admit a $P$-representation. The proof of this is  analogous to the derivation in quantum 
optics~\cite{scully} (see also~\cite{matthias:tomo} for a quantum information treatment).

\begin{lemma}\label{lem:hasPrepresentation}
	Any maximally symmetric matrix $M \in \maxsym( (\Real^n)^{\otimes \ell})$ admits a $P$-representation 
	\begin{align}
		M = \ints P_M(x) \proj{x}^{\otimes \ell} dx \ ,
	\end{align}
	for some function $P_M(x) $ on the sphere $S^{n-1}$.
\end{lemma}
\begin{proof}
	Note that the space of maximally symmetric operators that admits a $P$-representation is a linear subspace of $\maxsym( (\Real^n)^{\otimes \ell})$ defined as
	\begin{align}
		S_P := \left\{ V \in \maxsym((\Real^n)^{\otimes \ell})\mid \exists P_V \mbox{ s.t. } V = \ints P_V(x) \proj{x}^{\otimes \ell} dx\right\}\ .
	\end{align}
	Let $S_P^{\bot} := \{  W \in \maxsym( (\Real^n)^{\otimes \ell}) \mid \forall V \in S_P,\ \tr(V^T W) = 0\}$ denote its complement, i.e.
	the maximally symmetric matrices that do not admit a $P$-representation. 
	Our goal will be to show that the only maximally symmetric matrix $M \in S_P^{\bot}$ is $M = 0$. 
By defintion, we have that for any
	$M \in S_P^{\bot}$ and any $V \in S_P$
	\begin{align}\label{eq:givesZero}
		\tr\left(V^T M\right) &= \tr\left(VM\right) = \ints P_V(x) \bra{x}^{\otimes \ell}M\ket{x}^{\otimes \ell} dx\\
		&= \ints P_V(x) Q_M(x) dx = 0\ ,
	\end{align}
	where $Q_M(x) = \bra{x}^{\otimes \ell}M\ket{x}^{\otimes \ell}$ is the $Q$-representation of $M$. Note that 
	since $P_V(x)$ is an arbitrary square integrable  function on the hypersphere, ~\eqref{eq:givesZero} implies that $Q_M(x) = 0$ for all $x \in S^{n-1}$. Since $Q_M(x)$ is a polynomial of degree $2\ell$ we conclude that it is the zero polynomial. 
	Recall that there is a one-to-one mapping from polynomials to maximally symmetric matrices and the matrix $M$ for the zero polynomial is $M=0$.
\end{proof}

As before, we speak of a \emph{positive $P$-representation} if $P_M(x) \geq 0$ for all $x$.
Note that for $\tr(M) = 1$ we have
\begin{align}
1 = \tr(M) = \ints P_M(x) (\inp{x}{x})^{\ell} dx = \ints P_M(x) dx\ .
\end{align}
As such, we may again think of the numbers $P_M(x)dx$ as a quasi-measure, and indeed one exists for \emph{any} such $M$. 
The question whether the sequence of moments corresponding to $M$ admits a representing measure is thus equivalent to asking 
whether $M$ admits a \emph{positive} $P$-representation. 

\subsubsection{Evaluation on a polynomial matrix}\label{sec:realEval}

The following observation again leads to a nice interpretation 
of the solutions at intermediate (non-optimal) levels of the SDP hierarchy.
More precisely, using the $P$-representation we can write $\tr(M Z_T)$ where $M \geq 0$ with $M \in \Sym((\Complex^n)^{\otimes \ell})$ and $Z_T$ is the polynomial matrix of some polynomial $T(x)$ as
\begin{align}
\tr(M Z_T) 
&= \ints P_M(x) \bra{x}^{\otimes \ell} Z_T \ket{x}^{\otimes \ell} dx \nonumber \\
&= \ints P_M(x) T(x) dx  \label{eq:ppolyevalreal} \ .
\end{align}
That is, we can always understand this trace as the evaluation of the polynomial $T$ at points $x$ weighted by coefficients
$P_M(x)$. Again, as these are not necessarily positive at an intermediate level, $P_M(x) dx$ is not a measure. 

\subsubsection{Relation between the $P$ and $Q$-representation}
We are now ready to work out the relation between the $P$ and $Q$-representations of a maximally symmetric matrix $M$. From the definition of the $Q$-representation we see that
they are related as
\begin{eqnarray}
Q_M(x)&=&\bra{x}^{\otimes \ell} M\ket{x}^{\otimes \ell} = \bra{x}^{\otimes \ell} \left(\int_{S^{n-1}} P_M(y) \proj{y}^{\otimes \ell} dy\right) \ket{x}^{\otimes \ell} \\
&=&\int_{S^{n-1}} (\langle x|y\rangle)^{2\ell} P_M(y)dy. \label{eq:convolution}
\end{eqnarray}
That is, the Q-representation can be written as a convolution of the P-representation. Note that 
this operation can be seen to commute with the action of the orthogonal group on functions on the hypersphere. 
Furthermore, recall from Section~\ref{sec:spherical} 
we can decompose $P_M=\sum_{jm} p_{jm}^M s_{jm}(x)$ and $Q_M=\sum_{jm} q^M_{jm} s_{jm}(x)$ where $p_{jm}^M$ and $q_{jm}^M$ are the Fourier coefficients of $P_M$ and $Q_M$ respectively. Recall also that since $Q_M$ is a polynomial of order $2\ell$ only $j=0,2,...,2\ell$ occur in the sum over $j$. In essence, the following lemma
is an easy consequence of the Funk-Hecke formula~\eqref{eq:FunkHecke} as employed in Lemma~\ref{lem:sjmTrace}

\begin{lemma}\label{lem:PvsQ}
	Let $M \in \maxsym( (\Real^n)^{\otimes \ell})$ be a maximally symmetrix matrix with $P$-representation $M = \sum_{jm} p_{jm}^M S_{jm}^\ell$ 
	and $Q$-representation $Q_M(x) = \bra{x}^{\otimes \ell} M \ket{x}^{\otimes \ell} = \sum_{jm} q_{jm}^M s_{jm}(x)$. Then
	\begin{align}
		q_{jm}^M = \frac{\omega_{n-1}}{\omega_n}\ \lambda(n,\ell,j)\ p_{jm}^M\ ,
	\end{align}
	where $\lambda(n,\ell,j)$ is defined in Lemma~\ref{lem:simplifyLambda}.
\end{lemma}
\begin{proof}
	As the spherical harmonics are orthogonal, we can recover the Fourier coefficients of $Q_M(x)$ as
	\begin{align}
		\tr\left(\sjm M\right) &= \ints s_{jm}(x) \bra{x}^{\otimes \ell}M\ket{x}^{\otimes \ell} dx\\
		&= 
		\ints s_{jm}(x) Q_M(x) dx = \sum_{j'm'} q^M_{j'm'} \ints s_{j'm'}(x) s_{jm}(x) dx\\
		&= \frac{q_{jm}^M}{\omega_n}\ ,
	\end{align}
	where we have again used the fact that the trace is linear and cyclic.
	Writing $M$ in terms of its $P$-representation and using Lemma~\ref{lem:sjmTrace} we thus have
	\begin{align}
		q_{jm}^M = \omega_n \tr\left(\sjm M\right) = \omega_n \sum_{j'm'} p_{j'm'}^M \tr\left(\sjm \sjmt\right) = \frac{\omega_{n-1}}{\omega_n}\ \lambda(n,\ell,j)\ p_{jm}^M\ ,
	\end{align}
	which is our claim.
\end{proof}

The central idea of our main result will be that for $j\ll \ell$ we have that $p_{jm}^M$ is (up to normalization) almost equal to $q_{jm}^M$.
The following lemma makes it clear, which normalization we want to choose to relate the $P$ and $Q$ representations for any state $M$. This is done
by comparing the coefficients $p_{0}^M$ and $q_{0}^M$.

\begin{lemma}\label{lem:approxM}
	Let $M \in \maxsym( (\Real^n)^{\otimes \ell})$ be a maximally symmetric matrix satisfying $M \geq 0$ and $\tr(M) = 1$ with
$P$-representation $M = \sum_{jm} p_{jm}^M S_{jm}^\ell$ and
	$Q$-representation $Q_M(x) = \bra{x}^{\otimes \ell} M \ket{x}^{\otimes \ell} = \sum_{jm} q_{jm}^M s_{jm}(x)$. Then the matrix
	\begin{align}\label{eq:targetQ}
		\tilde{M} = \frac{\omega_n}{\omega_{n-1}\lambda(n,\ell,0)} \ints Q_M(x) \proj{x}^{\otimes \ell} dx
	\end{align}
	is satisfies $\tilde{M} \geq 0$ and $\tr(\tilde{M}) = 1$.
\end{lemma}
\begin{proof}
	Since $M \geq 0$, we have $Q_M(x) \geq 0$ for all $x \in S^{n-1}$ and hence $\tilde{M} \geq 0$. It remains to prove that $\tr(\tilde{M}) = 1$. 
	Using the $Q$-represention of $M$ and Lemma~\ref{lem:sjmTraceless} we can write the trace of the integral in~\eqref{eq:targetQ} as
	\begin{align}
		\tr\left(\sum_{jm} q_{jm}^M \ints s_{jm}(x) \proj{x}^{\otimes \ell} dx\right) = \sum_{jm} q_{jm}^M \tr\left(\sjm\right) 
		= \frac{q_{0}^M}{\sqrt{\omega_n}}\ .
	\end{align}
	Similarly, using the $P$-representation of $M$ and Lemma~\ref{lem:sjmTraceless}
	\begin{align}
		1 = \tr\left(M\right) = \frac{p_{0}^M}{\sqrt{\omega_n}}\ .
	\end{align}
	Using the relations between the $P$ and the $Q$-representation of a maximally symmetric matrix (Lemma~\ref{lem:PvsQ}) we thus
	have that
	\begin{align}
		\tr\left(\tilde{M}\right) = 
		\frac{p_{0}^M}{\sqrt{\omega_n}} = 1\ .
	\end{align}
\end{proof}

%
%
%
%
%

\section{The existence of approximate representing measures} \label{sec:deFinetti}

Our key result is a theorem on the existence of approximate representing measures for a set of possible moments given by a maximally symmetric matrix $M\in \mathcal{M}(a,\ell)$. Specifically we will show that if $M\in \mathcal{M}(a,\ell)$ then $M$ can be approximated by a maximally symmetric matrix $\tilde{M}$ that has a normalised positive P-function with an error that scales like $d^2/\ell$. 

Our main theorem was inspired by a set of results in quantum information that go by the name of finite quantum de Finetti theorems~\cite{matthias:definetti,robert:compendious,renato:symmetry,robert:thesis, robert:definetti1}. In its original form~\cite{definetti}, de Finetti's theorem states that
any exchangeable probability distribution over an infinite sequence of random variables $X_1,X_2,\ldots$ 
is equal to a convex sum of product distributions\footnote{A probability distribution is \emph{exchangeable} if for all permutations $\pi$ and all $x_1,x_2,\ldots$
the probability of a sequence and its permutation is the same, i.e., $P(x_1,x_2,\ldots) = P(x_{\pi(1)},x_{\pi(2)},\ldots)$.}.
Diaconis and Freedman~\cite{df} later showed that a similar statement still holds approximately if the distribution is not infinitely, but merely $\ell$-exchangeable. 
In particular, they showed that 
for any exchangeable probability distribution
$P_{X_1,\ldots,X_\ell}$ of $\ell$ random variables, the distribution on the first $X_1,\ldots,X_\red$ is closely approximated by a convex sum of product 
distributions. 
More precisely, 
there exists some measure $\mu$ on the set of distributions such that for all $x_1,\ldots,x_\red$
\begin{align}
	P_{X_1,\ldots,X_\red}(x_1,\ldots,x_\red) \approx \int P_X(x_1)\ldots P_X(x_\red) d\mu(P_X)\ ,
\end{align}
for some approximation parameter depending on $\red$ and $\ell$ and closeness is measured in terms of the statistical distance.
As quantum states can be understood as a generalization of classical probability distributions, De Finetti theorems for infinite~\cite{hudson,petz:definetti,caves:definetti,fuchs:definetti}
as well as finitely exchangeable quantum states~\cite{caves,matthias:definetti,robert:compendious,renato:symmetry,robert:thesis, robert:definetti1} 
have attracted much attention in quantum information. 

\subsection{For the case of complex variables}

\subsubsection{Statement and consequences}\label{sec:complexDefinetti}
We first state a finite de Finetti theorem for the complex case that is already known from quantum information. We note that it is in fact possible
to prove a complex de Finetti theorem in a rather different way along the lines of our real de Finetti below, but as this is likely only of interest to quantum
information theorists we omit an alternate proof~\cite{as:inprep}.
Instead, we simply state the following result proven in~\cite{matthias:definetti}, translated into a language that provides some insight
into the problem at hand.

\begin{theorem}[Unitary finite deFinetti Theorem~\cite{matthias:definetti}]\label{thm:unitarydeFinetti}
	Let $n,a,\ell \in \Natural$ and $a< \ell$. Let $M \in \symEll$ be a symmetric matrix that is a state (i.e., $\tr(M) = 1$ and $M \geq 0$), and let
	$Q_M(x) = \bra{x}^{\otimes \ell} M \ket{x}^{\otimes \ell}$ be its $Q$-representation. Define
	\begin{align}
		\tilde{M}_a := \frac{1}{\dim(\Sym( (\Complex^n)^{\otimes \ell}))} 
			\int Q_M(x) \proj{x}^{\otimes a} dx \ .
	\end{align}
	Then the reduced matrix $M_a := \tr_{\downarrow a }(M)$ is approximated by the matrix $\tilde{M}_a$ as
	\begin{align}
		\left\|M_a - \tilde{M}_a\right\|_{1} \leq \frac{a\ n^2}{\ell}\ .
	\end{align}
	Furthermore, $\tilde{M}_a$ is a state (i.e.,$\tr(\tilde{M}_a) \geq 0$, $\tr(\tilde{M}_a) = 1$), and has a positive $P$-representation.
\end{theorem}

A useful consequence of the de Finetti theorem is that by~\eqref{eq:nearlySameValue} 
for any matrices $Z_T$ with $\|Z_T\|_\infty \leq 1$ the difference between $\tr(Z_T M_1)$ and $\tr(Z_T M_2)$ is no more than $\eps$.
In particular, this means that if $Z_T$ is a polynomial matrix and thus $\tr(Z_T M_a)$ is the evaluation of the polynomial $T$ at points $x$
weighted by the cofficients of its $P$-representation $P_{M_a}(x)dx$ (see Section~\ref{sec:complexEval}), then 
the evaluation of $T$ at points weighted by $dm(x) = Q_M(x)dx/\dim(\Sym( (\Complex^n)^{\otimes \ell}))$
takes on almost the same value. Note that whereas $P_{M_a}(x)dx$ is \emph{not} a measure, $dm(x)$ actually is one.

\subsection{For the case of real variables}
We now proceed to prove a de Finetti theorem for real, maximally symmetric matrices. It should be noted that a de Finetti theorem for arbitrary real matrices
in the symmetric subspace cannot hold~\cite{caves}. In fact, this even lead to some speculation that de Finetti theorems may not be useful to study convergence
of hierarchies of semidefinite programs for polynomial optimisation~\cite{fernando:hierarchy}. However, as we saw earlier we do not care about arbitrary real matrices, as the relevant space is the space of
\emph{maximally} symmetric matrices. Indeed, the counterexample given in~\cite{caves} is real, but \emph{not} maximally symmetric. 
Intuitvely, the reason why our proof fails for general real symmetric (but not maximally symmetric) matrices is that the proof of the existence of a real
$P$-representation (Lemma~\ref{lem:PvsQ}) fails at the point where we need that $Q_M(x) = 0$ implies that $M = 0$, since general symmetric
matrices are \emph{not} equivalent to unique polynomials over real valued variables. 

\subsubsection{A de Finetti for polynomials}
We first prove a de Finetti theorem where the approximation is in terms of a norm that relates directly 
to the optimization of polynomials on the sphere.
More precisely, we will show approximation in terms of the \emph{F1}-norm on the space of
maximally symmetric matrices $A \in \maxsymEll$ defined as
\begin{align}
\|A\|_{F1} := \sup_{\|F\|_{\infty} \leq 1} \tr(Z_F A)\ ,
\end{align}
where the maximization is taken over homogeneous polynomials $F$ with polynomial matrix $Z_F$
and $\|F\|_{\infty} = \max_{x \in S^{n-1}} |F(x)|$ is the $p\rightarrow \infty$-norm for functions 
on the sphere. Claim~\ref{claim:F1isnorm} shows that this quantity is indeed a norm. We refer to the introduction for a high level description of this proof.

\begin{theorem}\label{thm:definettiF1}
	Let $n,a,\ell \in \Natural$ with $n \geq 3$, $a < \ell$ and $\ell \geq 2a^2(a + \frac{n}{2} - 1) - \frac{n}{2}$. Let $M \in \maxsymEll$ be a maximally symmetric matrix that is a state (i.e., $\tr(M) = 1$ and $M \geq 0$), and let
	$Q_M(x) = \bra{x}^{\otimes \ell} M \ket{x}^{\otimes \ell}$ be its $Q$-representation. Define
	\begin{align}
		\tilde{M}_a := \left(\frac{2^{2\ell} \omega_n}{\sqrt{\pi} \omega_{n-1}} \frac{\gam{\ell + 1}\gam{\ell + \frac{n}{2}}}{\gfrac{n-1}{2} \gam{2\ell + 1}}\right)
	\ints Q_M(x) \proj{x}^{\otimes a} dx \ .
	\end{align}
	Then the reduced matrix $M_a := \tr_{\downarrow a}(M)$ is approximated by the matrix $\tilde{M}_a$ as
	\begin{align}
		\left\|M_a - \tilde{M}_a\right\|_{F1} \leq \frac{4a^2\left(a + \frac{n}{2} - 1\right)}{2\ell + n}\ .
	\end{align}
	Furthermore, $\tilde{M}_a$ is a state (i.e., $\tr(\tilde{M}_a) \geq 0$, $\tr(\tilde{M}_a) = 1$), and has a positive $P$-representation.
\end{theorem}
\begin{proof}
We first show that $\tilde{M}_a$ is a state.
Using the fact that $M \geq 0$ and $\tr(M)=1$ it is not hard to see (Lemma~\ref{lem:approxM}) that the matrix
\begin{align}
\tilde{M} = \frac{\omega_n}{\omega_{n-1}\lambda(n,\ell,0)} \ints Q_M(x) \proj{x}^{\otimes \ell} dx\ ,
\end{align}
where $\lambda(n,\ell,j)$ is defined as in Lemma~\ref{lem:simplifyLambda}, satisfies $\tilde{M} \geq 0$ and $\tr(\tilde{M}) = 1$. Hence
\begin{align}
	\tilde{M}_a = \tr(\tilde{M}) = \frac{\omega_n}{\omega_{n-1}\lambda(n,\ell,0)} \ints Q_M(x) \proj{x}^{\otimes a} dx
\end{align}
satisfies $\tilde{M}_a \geq 0$ and $\tr(\tilde{M}_a) = 1$. Note that $Q_M(x)$ is positive because $M \geq 0$. 

Let us now prove the claimed upper bound. In terms of the Fourier coefficients, we can write $\tilde{M}_a$ in its $P$-representation
as 
\begin{align}
\tilde{M}_a = \int P_{\tilde{M}_a}(x) \proj{x}^{\otimes a} dx\ , 
\end{align}
where
\begin{align}\label{eq:prepApprox}
 P_{\tilde{M}_a}(x) = \frac{\omega_n}{\omega_{n-1} \lambda(n,\ell,0)} \sum_{jm} q_{jm}^M s_{jm}(x)\ .
\end{align}
Similarly, since $M$ is a maximally symmetric matrix, it admits a $P$-representation (Lemma~\ref{lem:hasPrepresentation}) as 
\begin{align}
M = \ints P_M(x) \proj{x}^{\otimes \ell} dx = \sum_{jm} p_{jm}^M \sjm\ .
\end{align}
Given its $P$-representation we can immediately write down the reduced matrix as
\begin{align}\label{eq:pForm}
M_a = \ints P_{M_a}(x) \proj{x}^{\otimes a} dx = \sum_{jm} p_{jm}^M S_{jm}^a\ .
\end{align}

Our goal is to show that the low order
Fourier coefficients $p_{jm}^M$ stemming from the $P$-representation are very closely related to those of the $Q$-representation.
For $Z_F = \sum_{jm} p_{jm}^F S_{jm}^a$, let $Z_F^j = \sum_m p_{jm}^F S_{jm}^a$.
For any $Z_F \in \maxsymR$ we can write 
\begin{align}
&\left|\tr\left(Z_F(M_a - \tilde{M}_a)\right)\right|\\
&\qquad=
\left|\sum_{j=0}^{2a} \sum_m p^F_{jm} \sum_{j'm'} \left(p^M_{j'm'} - \frac{\omega_n}{\omega_{n-1}\lambda(n,\ell,0)} q^{M}_{j'm'}\right)
\tr\left(S_{jm}^a S_{j'm'}^a\right)\right|\nonumber \\
&\qquad=
\left|\sum_{j=0}^{2a} \sum_m p^F_{jm} \left(p^M_{jm} - \frac{\omega_n}{\omega_{n-1}\lambda(n,\ell,0)} q^{M}_{jm}\right) \tr\left(S_{jm}^a S_{jm}^a\right)\right| \\
&\qquad=
\left|\sum_{j=0}^{2a} \sum_m p^F_{jm} \left(\frac{\omega_{n}}{\omega_{n-1}\lambda(n,\ell,j)} - 
\frac{\omega_n}{\omega_{n-1}\lambda(n,\ell,0)}\right) 
q^{M}_{jm} \tr\left(S_{jm}^a S_{jm}^a\right)\label{eq:stepp1}\right| \\
&\qquad=
\left|\sum_{j=0}^{2a} \left(\frac{\lambda(n,\ell,0)}{\lambda(n,\ell,j)}-1\right)
\left(\frac{\omega_n}{\omega_{n-1}\lambda(n,\ell,0)} \right)
 \sum_m p^{F}_{jm}
q^{M}_{jm} \tr\left(S_{jm}^a S_{jm}^a\right)\label{eq:stepp2}\right| \\
&\qquad=
\left|\sum_{j=0}^{2a} \left(\frac{\lambda(n,\ell,0)}{\lambda(n,\ell,j)}-1\right)\tr(Z_F^j \tilde{M}_a)\right|
\end{align}
where ~\eqref{eq:stepp1} follows from the orthogonality of the spherical harmonic
matrices $S_{jm}^r$ and $S_{j'm'}^r$ for $j\neq j'$ and $m \neq m'$ (Lemma~\ref{lem:sjmTrace}),
\eqref{eq:stepp2} from the relation between the $P$ and the $Q$-representation (Lemma~\ref{lem:PvsQ}),
and the last equality again from the orghonality of the spherical harmonics.
We can thus write the norm as
\begin{align}
\|M_a - \tilde{M}_a\|_{F1} &= \sup_{\|F\|_{\infty} \leq 1} \tr(Z_F(M_a - \tilde{M}_a)\\
&\leq 
\sum_{j=0}^{2a} \left|\left(\frac{\lambda(n,\ell,0)}{\lambda(n,\ell,j)}-1\right)\right| \|\tilde{M}_a\|_{F1}\\
&\leq 
\frac{4a^2\left(a + \frac{n}{2} - 1\right)}{2\ell + n}
\|\tilde{M}_a\|_{F1}\ .
\end{align}
where the first inequality follows from the fact that $\|\cdot\|_{F1}$ is a norm
and the second from Corollary~\ref{cor:epsilonBound}.
It remains to bound $\|\tilde{M}_a\|_{F1} = \sup_{\|F\|_{\infty} \leq 1} |\tr(\tilde{M}_a Z_F)|$.
Using the fact that for all $x$ $Q_M(x) \geq 0$ and $\omega_n/(\omega_{n-1}\lambda(n,\ell,0)) \geq 0$ we can bound
\begin{align}
|\tr(\tilde{M}_a Z_F)| 
&= \left|\frac{\omega_n}{\omega_{n-1}\lambda(n,\ell,0)} \ints Q_M(x) F(x) dx\right|\\
&\leq \frac{\omega_n}{\omega_{n-1}\lambda(n,\ell,0)} \ints Q_M(x) |F(x)|  dx\\
&\leq \frac{\omega_n}{\omega_{n-1}\lambda(n,\ell,0)} \ints Q_M(x)  dx\\
&=\tr(\tilde{M}_a) = 1\ ,
\end{align}
which yields our claim.
\end{proof}

\subsubsection{A standard de Finetti}
For our convergence results this is all we need.
For completeness, however, we state a real de Finetti theorem in its usual form where the norm
is not related the evaluation of polynomials but related directly to matrices. 
For our proof, let us first establish the following lemma.

\begin{lemma}\label{thm:realMetaFinettiHarmonic}
	Let $n,a,\ell,j \in \Natural$ with $n \geq 3$, $a < \ell$, and even $j \leq 2a$. Let $M \in \maxsymEll$ be a maximally symmetric matrix that is a state (i.e., $\tr(M) = 1$ and $M \geq 0$), and let
	$Q_M(x) = \bra{x}^{\otimes \ell} M \ket{x}^{\otimes \ell}$ be its $Q$-representation. Define
	\begin{align}
		\tilde{M}_a := \left(\frac{2^{2\ell} \omega_n}{\sqrt{\pi} \omega_{n-1}} \frac{\gam{\ell + 1}\gam{\ell + \frac{n}{2}}}{\gfrac{n-1}{2} \gam{2\ell + 1}}\right)
	\ints Q_M(x) \proj{x}^{\otimes a} dx \ ,
	\end{align}
	and the reduced matrix $M_a := \tr_{\downarrow a}(M)$. Then any for harmonic homogeneous polynomial $F$ of degree $j$, i.e., 
	for any maximally symmetric matrix $Z_F \in \maxsymR$ of the form $Z_F = \sum_m p_{jm}^F S_{jm}^a$,
	\begin{align}
		\left|\tr\left[Z_F\left(M_a - \tilde{M}_a\right)\right]\right| \leq 
		\frac{j\left(\frac{j+n}{2} -1\right)}{2\ell+ n}\ \left|\tr(Z_F M_a)\right|\ .
	\end{align}
	Furthermore, $\tilde{M}_a$ is a state (i.e.,$\tr(\tilde{M}_a) \geq 0$, $\tr(\tilde{M}_a) = 1$), and has a positive $P$-representation.
\end{lemma}
\begin{proof}
The fact that $\tilde{M}_a$ forms a state has been shown in the proof of Theorem~\ref{thm:definettiF1}.
Recall also that $M_a$ and $\tilde{M}_a$ can be written in terms of the P representation of $M$.
Since $F$ is harmonic homogeneous of degree $j$ its polynomial matrix $Z_F$ only has components for this $j$, i.e.,
its $P$-representation is of the form
\begin{align}
Z_F = \sum_{m=1}^{N(n,j)} p_{jm}^F S_{jm}^a\ .
\end{align}
We thus have
\begin{align}
&\tr\left(Z_F(M_a - \tilde{M}_a)\right)\\
&\qquad= 
\sum_m p^F_{jm} \sum_{j'm'} \left(p^M_{j'm'} - \frac{\omega_n}{\omega_{n-1}\lambda(n,\ell,0)} q^{M}_{j'm'}\right) 
\tr\left(S_{jm}^a S_{j'm'}^a\right)\nonumber \\
&\qquad= 
\sum_m p^F_{jm} \left(p^M_{jm} - \frac{\omega_n}{\omega_{n-1}\lambda(n,\ell,0)} q^{M}_{jm}\right) \tr\left(S_{jm}^a S_{jm}^a\right)\label{eq:step1} \\
&\qquad=
\left(1 - \frac{\lambda(n,\ell,j)}{\lambda(n,\ell,0)}\right) \sum_m p^{F}_{jm}
p_{jm}^M  \tr\left(S_{jm}^a S_{jm}^a\right) \label{eq:step2}\\
&\qquad = 
\left(1 - \frac{\lambda(n,\ell,j)}{\lambda(n,\ell,0)}\right) \tr(Z_F M_a)\ ,\label{eq:step3}
\end{align}
where ~\eqref{eq:step1} and~\eqref{eq:step2} follow 
from the orthogonality of the matrices $S_{jm}^r$ and $S_{j'm'}^r$ for $j\neq j'$ and $m \neq m'$ (Lemma~\ref{lem:sjmTrace}),
and \eqref{eq:step2} from the relation between the $P$ and the $Q$-representation (Lemma~\ref{lem:PvsQ}).
It remains to bound~\eqref{eq:step3}. Lemma~\ref{lem:epsilonBound} shows that for any even $j$
\begin{align}
1 - \frac{\lambda(n,\ell,j)}{\lambda(n,\ell,0)}
\leq 
\frac{j\left(\frac{j+n}{2} -1\right)}{2\ell+ n}\ .
\end{align}
which is our claim.
\end{proof}

We are now ready to prove a real de Finetti theorem for the matrix L1-norm.

\begin{theorem}\label{thm:definettiMS1}
	Let $n,a,\ell \in \Natural$ with $n \geq 3$ and $a < \ell$. Let $M \in \maxsymEll$ be a maximally symmetric matrix that is a state (i.e., $\tr(M) = 1$ and $M \geq 0$), and let
	$Q_M(x) = \bra{x}^{\otimes \ell} M \ket{x}^{\otimes \ell}$ be its $Q$-representation. Define
	\begin{align}
		\tilde{M}_a := \left(\frac{2^{2\ell} \omega_n}{\sqrt{\pi} \omega_{n-1}} \frac{\gam{\ell + 1}\gam{\ell + \frac{n}{2}}}{\gfrac{n-1}{2} \gam{2\ell + 1}}\right)
	\ints Q_M(x) \proj{x}^{\otimes a} dx \ .
	\end{align}
	Then the reduced matrix $M_a := \tr_{\downarrow a}(M)$ is approximated by the matrix $\tilde{M}_a$ as
	\begin{align}
		\left\|M_a - \tilde{M}_a\right\|_{1} \leq \frac{2a^2\left(a + \frac{n}{2} - 1\right)}{2\ell + n}\ .
	\end{align}
	Furthermore, $\tilde{M}_a$ is a state (i.e.,$\tr(\tilde{M}_a) \geq 0$, $\tr(\tilde{M}_a) = 1$), and has a positive $P$-representation.
\end{theorem}
\begin{proof}
	We want to show that the distance between $M_a$ and $\tilde{M}_a$ is bounded.
By the definition of the L1-norm 
\begin{align}\label{eq:oneCor}
	\left\|M_a - \tilde{M}_a\right\|_{1} = \sup_{\|\bar{Z}_F\|_{\infty} \leq 1} \tr\left(\bar{Z}_F (M_a - \tilde{M}_a)\right)\ ,
\end{align}
where the maximum is taken over all $\bar{Z}_F \in \End(R^{\otimes a})$.
Recall from the proof of Lemma~\ref{lem:PvsQ} that the space of maximally symmetric matrices forms a linear subspace. As such, 
we can write $\bar{Z}_F = Z_F + Z_F^{\perp}$ where $Z_F \in \maxsymR$ and $\tr(Z_F^{\perp} S) = 0$ for any maximally symmetric matrix $S \in \maxsymR$.
Furthermore, note that we can expand the maximally symmetric component as
\begin{align}
Z_F = \sum_{j=0}^{2a} Z_F^j\ ,
\end{align}
where $Z_F^j$ is the polynomial matrix of a harmonic homogeneous polynomial of degree $j$. 
For each individual $j$ we can bound
\begin{align}
\sup_{\|Z_F^j\|_{\infty} \leq 1}  \tr\left(Z_F^j (M_a - \tilde{M}_a)\right)
&\leq
\frac{j\left(\frac{j+n}{2} -1\right)}{2\ell+ n} 
\sup_{\|Z_F^j\|_{\infty} \leq 1}  
\tr(Z_F^j M_a)\\
&\leq 
\frac{j\left(\frac{j+n}{2} -1\right)}{2\ell+ n}\ , \label{eq:individualj}
\end{align}
where the first inequality follows from Lemma~\ref{thm:realMetaFinettiHarmonic}, and the second
from the fact that $M_a \geq 0$, $\tr(M_a) = 1$ and hence $\tr(Z_F^j M_a) \leq 1$.
Using that $\tr(Z_F^{\perp} M_a) = \tr(Z_F^{\perp} \tilde{M}_a) =0$ we can thus bound
\begin{align}
	\sup_{\|\bar{Z}_F\|_{\infty} \leq 1} \tr\left(Z_F (M_a - \tilde{M}_a)\right)
&=
\sup_{\|\bar{Z}_F\|_{\infty} \leq 1} \sum_j \tr\left(Z_F^j (M_a - \tilde{M}_a)\right)\\
& \leq 
\sum_j \sup_{\|Z_F^j\|_{\infty} \leq 1} \tr\left(Z_F^j (M_a - \tilde{M}_a)\right)\\
&\leq
\sum_j \frac{j\left(\frac{j+n}{2} -1\right)}{2\ell+ n} \ ,
\end{align}
where the first inequality comes from the convexity of the norm, and the second from~\eqref{eq:individualj}.
Our claim now follows by noting that the r.h.s is increasing in $j$ and only even $j$ with $j \leq 2a$ appear in the sum.
\end{proof}

\section{Hierarchy of semidefinite programming relaxations}\label{sec:using}

We are now finally ready to return to our problem of optimizing $T(x)$ over the sphere.
Recall from Section~\ref{sec:evenoddDegree} that it is sufficient for us to understand how to optimize homogeneous polynomials $T(x)$ of \emph{even} degree $d=2a$.
This will be done using a hierarchy of semi-definite programs of the kind introduced for polynomial optimization problems by Parrilo~\cite{parrilo} and Laserre~\cite{lasserre}. We will determine the performance of our hierarchy of optimizations at a given level of the hierarchy. We will just describe the case of the hypersphere, although the result for the complex projective space is straightforward. 

\subsection{Semidefinite programming relaxations of polynomial optimization}

The key to the semidefinite programming relaxations of~\cite{lasserre,parrilo} is to note that one can formulate our optimization problem in terms of truncated moment matrices.
More precisely, the problem of optimizing $T(x)$ over the sphere can be cast as an equivalent problem of optimizing
\begin{sdp}{maximize}{$\tr(Z_T M_a)$}
	& $M_a \in \mathcal{M}_{+}(a)$\ ,
\end{sdp}
where $Z_T$ is the polynomial matrix of $T(x)$ as defined in Section~\ref{sec:realpoly} and $\mathcal{M}_{+}(a)$ is the set of truncated moment matrices from Section~\ref{sec:momentsreal}.
Unfortunately, this reformulation does not by itself enable us to solve the problem since characterizing the set of true truncated moment matrices $\mathcal{M}_{+}(a)$ is 
very difficult. Yet, this rewriting suggests a very natural hierarchy of semidefinite programs in which we relax the constraint that $M_a \in \mathcal{M}_{+}(a)$. 

As we have seen in Section~\ref{sec:momentsreal}, we can define interesting sets of matrices $\mathcal{M}(a,\ell)$ for any $\ell \geq a$ 
that contain all the truncated moment matrices in $\mathcal{M}_{+}(a)$. Unlike $\mathcal{M}_{+}(a)$, checking membership of $\mathcal{M}(a,\ell)$ is a semidefinite programming feasibility problem and is thus computationally tractable. 
This observation will allow us to write semidefinite programs that place upper bounds on our polynomial optimization. 
To see that these bounds are actually meaningful, note that 
we know from the de Finetti theorem of Section~\ref{sec:deFinetti} that for large $\ell$ there does exist a measure $\mu$ such that $\tilde{M}_a(\mu)\approx M_a$ when $M_a\in \mathcal{M}(a,\ell)$.
That is, there exists a true truncated moment matrix $\tilde{M}_a(\mu) \in \mathcal{M}_{+}(a)$ that closely approximates $M_a$.
This second observation will allow us to show that the optimum of these semidefinite programs is close to the true optimum $\nu$ for large $\ell$.

We define our level $\ell$ relaxation of the polynomial optimization problem as follows
\begin{sdp}{maximize}{${\rm Tr}(Z_TM_a)$}
	& $M_a \in \mathcal{M} (a,\ell)$\ ,\\
\end{sdp}
where $\mathcal{M}(a,\ell)$ is defined in Section \ref{sec:momentsreal}. 
We will refer to the optimum value of this semidefinite program at level $\ell$ as $\nu_\ell$ and use $M^*_a$ to denote an optimal solution, i.e.,
$\tr(Z_TM_a^*) = \nu_\ell$. Note that any feasible solution $M_a$ at level $\ell$ admits a very nice operational interpretation due to the fact that
any maximally symmetric matrix $M_a$ admits a $P$-representation (see Section~\ref{sec:PvsQ}).
More precisely, 
recall that from equation~(\ref{eq:ppolyevalreal}) we have ${\rm Tr}(Z_TM_a)=\int_{S^{n-1}}T(x)P_{M_a}(x)dx$. That is, the term $\tr(Z_F M_a)$ corresponds
to an evaluation of the polynomial $T(x)$ at points $x$ on the sphere weighted by $P_{M_a}(x)dx$. These weights are not necessarily positive but 
if $M_a$ were actually a truncated moment matrix $M_a\in \mathcal{M}_{+}(a)$ then $P_{M_a}(x)dx$ is a normalised probability measure on the hypersphere. 
Since $\mathcal{M}_{+}(a) \subset \mathcal{M}(a,\ell)$ we have
\begin{equation}
\nu_\ell \geq \nu
\end{equation}
and so $\nu_\ell$ is a global upper bound on the polynomial $T(x)$. This upper bound on $\nu$ was originally emphasised by Laserre and Parrilo~\cite{parrilo,lasserre}.

Let us now write our relaxation explicitly as a semidefinite program based on the definition of $\mathcal{M}(a,\ell)$.
Define the degree $2\ell$ polynomial $T'(x)=r^{2(\ell-d)}(x)T(x)$ and let $Z_{T'}$ denote its polynomial matrix.  
Introducing the variable $M\in \End(\Sym(V^{\otimes \ell}))$,
we can then write our relaxations explicitly as
\begin{sdp}{maximize}{${\rm Tr}(Z_{T'}M)$}
	& $M \geq 0 $\ , \\
	& $\tr(M) = 1  $\ , \\
	& $(I-\Pi_{2\ell})|M\rangle = 0$\ .
\end{sdp}
To see the relation to the previous formulation, note that the conditions on $M$ guarantee that $M_a= \tr_{\downarrow a}(M)\in \mathcal{M}(a,\ell)$ which follows from the definition in Section \ref{sec:momentsreal}. On the one hand for all $M_a\in \mathcal{M}(a,\ell)$ there exists some $M\in \maxsymEll$ such that $M_a = \tr_{\downarrow a}(M)$. 
Thus the feasibility set of the semidefinite program corresponds to the feasibility set $\mathcal{M}(a,\ell)$ of our relaxation. On the other hand equation (\ref{eq:realpartial}) implies that $\tr(Z_{T'}M)=\tr(Z_{T}M_a)$ and so the objective function of the two optimizations are also identical. 

The set of semidefinite programming relaxations correspond to the moment matrix methods of Laserre~\cite{lasserre}. The dual semidefinite program~\cite{boyd:book} has an interpretation in terms of an optimization over sum of squares of decompositions~\cite{parrilo}. It is straightforward to compute this dual semidefinite program. 
The variables are a scalar $t$ and a matrix $\bar{Z}\in \End(\Sym(V^{\otimes \ell}))$ and the SDP is as follows
\begin{sdp}{minimize}{$t$}
	& $tI-Z_{T'}+\bar{Z} \geq 0 $\ , \\
	& $\Pi_{2\ell}|\bar{Z}\rangle = 0$ \ .
\end{sdp}
Comparing the conditions of this dual SDP to the sets of polynomials having sum of squares decompositions introduced in Section \ref{sec:sosreal} we find that we may interpret this semi-definite program as follows
\begin{sdp}{minimize}{$t$}
	& $ t-T(x) \in \mathcal{S} (a,\ell)$
\end{sdp}
which corresponds to the well-known duality between the Laserre and Parrilo points of view. This set of relaxations for the problem of optimising homogenous polynomials of even degree over the sphere is discussed in~\cite{fayb}.
This dual semidefinite program can be seen to be strictly feasible~\cite{boyd:book}, since we can choose a feasible point $\bar{t},\bar{Z}$ where $\bar{Z}=0$  and $\bar{t}>0$ is larger than the magnitude of the largest negative eigenvalue of $Z_{T'}$. At this feasible point we have $tI-Z_{T'}+\bar{Z} > 0$ and thus the semidefinite program is strictly feasible. As a result the optimum of the dual semidefinite program is equal to the primal optimum $\nu_\ell$~\cite[Theorem 3.1]{boyd:book}. 

\subsubsection{Complexity of the relaxation}

We will briefly consider the computation required to implement our relaxation at level $\ell$, this discussion closely matches~\cite{parrilo}. Note that the matrices act on the space  $\Sym((\Real^n)^{\otimes \ell})$ and therefore have dimension
\begin{align}
p={\ell + n  - 1 \choose \ell} \leq (\ell + 1)^n\ .
\end{align}
We have written the primal SDP of our relaxation in the so-called {\it normal form}~\cite{boyd:book}, which involves linear constraints on a matrix variable $M$. These constraints project the variable $M$, which has   $(p^2+p)/2$ independent entries, onto a subspace of dimension ${2\ell + n  - 1 \choose 2\ell}$ that is of much smaller dimension. As a result, for numerical implementation one would solve these matrix equalities to end up with a semidefinite program in the so-called {\it inequality form}~\cite{boyd:book} with $q={2\ell + n  - 1 \choose 2\ell}$ scalar variables. The complexity of obtaining approximate solutions to the SDP then scales no worse than $p^{5/2}q^2$~\cite{boyd:book,vandenbergheboyd}. For a fixed dimension of the hypersphere $n$ this scaling is polynomial in $\ell$ and of order $\ell^{9n/2}$. 
If one instead regards the level $\ell$ of the hierarchy as a constant, i.e., 
then the scaling with the number of variables $n$ is polynomial in $\ell$. We have $p\leq n^{\ell+1}$ and $q\leq n^{2\ell+1}$ and so the overall scaling is  $n^{13\ell/2+9/2}$. However if both $n$ and $\ell$ are growing with problem size then the computation require does grow exponentially.

\subsection{Performance of the relaxation}
Having solved the relaxation at level $\ell$ we would like to know how close we are to the optimal solution $\nu$.
More precisely, having obtained an optimal matrix $M^*_a\in \mathcal{M}(a,\ell)$ such that $\tr(M^*_aZ_T)=\nu_\ell$, we would like to know how close $\nu_\ell$ is to $\nu$. 
The key to relating $\nu_\ell$ to $\nu$ is given by our real valued de Finetti theorem (Theorem~\ref{thm:definettiF1}).
This theorem implies that for large $\ell$, the matrix $M^*_a$ is very close to a truncated moment matrix $\tilde{M}^*_a(\mu)$ 
and as a result must be close to an achievable value of $T(x)$. Specifically there exists a representing probability measure $\mu(x) dx$ and a corresponding moment matrix $\tilde{M}^*_a(\mu)$ that is approximately equal to $M_a$. From equation (\ref{eq:polyav}) we have for sufficiently large $\ell$ that
\begin{equation}
\nu\geq \tilde{\nu}_\ell = \int_{S^{n-1}} T(x) d\mu(x) = \tr(\tilde{M}^*_a(\mu)Z_T) \approx \tr(M^*_a Z_T)=\nu_\ell.
\end{equation}
Since $\nu_\ell\geq \nu$ for all levels of the relaxation, we can conclude that for large $\ell$ we have $\nu_\ell \approx \nu$. 
What is more, this measure can be determined explicitly in terms of the Fourier coefficients of the $P$-representation of $M_a^*$. Concretely, 
the Fourier coefficient $p_{jm}$ of the $P$-representation can be found
by computing $\tr(M_a^* S_{jm})$. 
In the following we derive an explicit bound on $\nu$ resulting from $\nu_\ell$. The condition on $\ell$ ensures that the approximation parameter $\epsilon(a,\ell,n) \leq 1$.

\begin{theorem}\label{thm:convergence}
	Let $\ell \geq 2a^2 \left(a + \frac{n}{2} - 1\right) - \frac{n}{2}$, and let $T(x)$ be a homogeneous polynomial of degree $2a$ with $x \in S^{n-1}$.
	Let $\nu_\ell$ be the optimum of the SDP relaxation at level $\ell$ and $\nu$ the true optimum of the homogenous polynomial $T(x)$ over the hypersphere $S^{n-1}$. 
Then the relative error in approximation is
\begin{align}
\label{eq:relerror}
		\frac{|\nu_\ell-\nu|}{|\nu|}\leq \epsilon(a,\ell,n)\ , 
	\end{align}
where
\begin{align}
		\epsilon(a,\ell,n) =  
 		\left(\frac{4a^2\left(a + \frac{n}{2} -1\right)}{2\ell+ n}\right)\ .
\end{align}
\end{theorem}
\begin{proof}
	The upper bound on $\nu$, $\nu\leq \nu_\ell$ is true for all $\ell$.
	Since $\nu_\ell$ is a solution of the semidefinite program there is an $M^*_a\in \mathcal{M}(a,\ell)$ with $\tr(Z_TM_a^*)=\nu_\ell$. 
Noting that $\|T\|_{\infty} = |\nu|$ we have by Theorem \ref{thm:definettiF1} that there exists a measure $\mu(x)dx$ on $S^{n-1}$ normalized to unity, and an associated truncated moment matrix $\tilde{M}^*_a(\mu)$ such that
\begin{equation}
\label{eq:definetti1}
|\nu_\ell-\tilde{\nu}_\ell|=\left|\tr\left[Z_T\left(M_a^* - \tilde{M}^*_a(\mu)\right)\right] \right|\leq 
\epsilon(a,\ell,n)|\nu|
\end{equation}
where $\epsilon(a,\ell,n)$ is as defined above and $\tilde{\nu}_\ell = \int_{S^{n-1}} T(x) \mu(x)dx=\tr(\tilde{M}^*_a(\mu)Z_T)$ (recall~\eqref{eq:polyav}). 
We thus have
\begin{equation}
\label{eq:bound}
\tilde{\nu}_\ell = \int_{S^{n-1}} T(x) \mu(x)dx \leq \nu \int_{S^{n-1}}\mu(x)dx = \nu \leq \nu_\ell\ ,
\end{equation}
where the second equality is just the normalisation of the measure.
As a result, from equation (\ref{eq:definetti1}) and~\eqref{eq:bound} we have
\begin{equation}
0 \leq \nu_\ell - \nu \leq \nu_\ell-\tilde{\nu}_\ell \leq \epsilon(a,\ell,n) |\nu| \ ,
\end{equation}
which is gives the relative error~\eqref{eq:relerror}. 
	\end{proof}

This result allows us to approximate the maximum of the polynomial $T(x)$ arbitrarily accurately by increasing $\ell$. We will now discuss the quality of this optimisation in terms more closely related to the standard conventions in the literature~\cite{deklerk}. In particular it is usual to measure the accuracy of approximation relative to $\nu-\nu_{\rm min}$ where $\nu_{\rm min}$ is the minimum value of $T(x)$ on the hypersphere. Specifically, we will consider only homogenous polynomials of even degree in order to compare to~\cite{fayb,reznik}. In that case we know that $T(0)=0$ and hence $\nu_{\rm min}\leq 0$ and so the bound on the relative error proven above is sufficient to obtain bounds on the usual measure of error as follows (since $\nu\leq \nu- \nu_{\rm min}$)
\begin{corollary}\label{cor:convergence}
	Let $\ell > \ell_0=2a^2 \left(a + \frac{n}{2} - 1\right) - \frac{n}{2}$, and let $T(x)$ be a homogeneous polynomial of degree $2a$ with $x \in S^{n-1}$.
	Let $\nu_\ell$ be the optimum of the SDP relaxation at level $\ell$ and $\nu$ the maximum of the homogenous polynomial $T(x)$ over the hypersphere $S^{n-1}$ and $\nu_{\rm min}$ be the minimum.
Then the error in approximation for the relaxation at level $\ell$ satisfies
\begin{align}
		|\nu_\ell-\nu| \leq \epsilon(a,\ell,n) (\nu-\nu_{\rm min})\ , 
	\end{align}
where
\begin{align}
		\epsilon(a,\ell,n) =  
 		\left(\frac{4a^2\left(a + \frac{n}{2} -1\right)}{2\ell+ n}\right)\ .
\end{align}
\end{corollary}
Thus the extension at level $\ell$ provides a $(1-\epsilon)$-approximation for $\nu$ in the sense of~\cite{deklerk}. This result corresponds precisely to Theorem 1 of~\cite{fayb} who proved it with the parameters $\ell_0=na(2a-1)/2\log 2 - n/2 -a$ and $\epsilon(a,\ell,n)=2an(2a-1)/[4\log2(l+n/2+a)-2an(2a-1)]$ using the main theorem of~\cite{reznik}. These parameters are very comparable, so for example in the regime where $\ell\gg n\gg a$ we find $\ell_0\simeq a^2n$ and $\epsilon \simeq a^2n/\ell$ while~\cite{fayb} has $\ell_0\simeq a^2n/\log 2$ and $\epsilon\simeq a^2 n/\log 2 \ell$ which is essentially the same. 

Notice that if we fix the accuracy $\epsilon$ to which we wish to work, and the degree of the polynomial $d=2a$ then we must use $\ell \simeq a^2n/\epsilon$ to achieve this accuracy. Since the required $\ell$ grows with $n$ the computation required to perform the necessary semidefinite program grows exponentially with $n$. Specifically we find that $p\simeq (a^2/\epsilon)^n$ and $q\simeq (2a^2/\epsilon)^n $  so the overall computation required scales like $2^{2n}(a^2/\epsilon)^{9n/2}$. As a result this is not a polynomial time approximation  scheme (PTAS)~\cite{deklerk}. As far as we know the question of whether such an approximation scheme exists for this set of problems is open. We emphasize, however, that our method has a number of advantages as outlined below.

\section{Discussion}

We have shown convergence of the SDP hierarchy of~\cite{lasserre,parrilo} for optimizing homogeneous polynomials on the hypersphere. 
The key ingredient was a so-called real valued de Finetti theorem, inspired by tools used in quantum information theory. 
Our approach allowed us to intepret the actual solution of the SDP hierarchy at level $\ell$ as the evaluation of the polynomial $T(x)$ at an affine
combination of points on the sphere. This combination was determined by the so-called $P$-representation $P_M(x)$. 
We could then find an explicit representing measure (i.e., a convex combination of points on the sphere) 
by comparing the Fourier coefficients of $P_M(x)$ with those of the so-called $Q$-representation $Q_M(x)$ which is always positive. 
Whereas our approach may seem somewhat foreign at first glance compared
to prior approaches, we find it rather beautiful in that it not only allowed us to interpret
solutions at any level $\ell$, but also gave a very natural way to obtain an approximate representing measure.
Relations between such $P$ and $Q$-representations are employed in quantum optics, which served as one of the inspirations of our proof.

It is interesting to note that our approach bears similarities to the proof by Laurent~\cite{monique} of the result of Curto and Fialkow
for the case of $\ell \rightarrow \infty$, and partially inspired our proof. In particular, one can loosely understand
part of the approach of~\cite{monique} as a ``measurement'' of the (infinite) moment matrix by interpolation polynomials. In spirit, this is loosely similar
to the ``measurement'' that we perform here to define the $Q$-function $Q_M(x)$. 

One may wonder whether other de Finetti theorems can be proven this way. It turns out that this is indeed the case, for example it is possible to recover the
usual quantum mechanical de Finetti theorem of~\cite{matthias:definetti} with similar approximation parameters using our techniques based on spherical harmonics.
Using the appropriate spherical harmonics for complex projective spaces an analogue of the Funk-Hecke formula~\cite{takeuchi} can be
found. 
It should be noted that the necessary convolution theorems analogous to the Funk-Hecke formula can in principle also be found for many other Riemannian symmetric 
spaces following~\cite{takeuchi}, opening the door to a whole new class of de Finetti theorems.
In turn, this would enable a proof of convergence for the optimization of polynomials over other Riemannian symmetric spaces.

\bibliographystyle{amsplain}
\bibliography{q}

\appendix

In this appendix we provide some of the more technical components of the real De Finetti theorem and our convergence results.

\section{Technical lemmas}\label{app:technical}

In our proof, we will need properties of spherical harmonics expressed with the help of so-called Gegenbauer polynomials $P_j: \Real \rightarrow \Real$. These can be expressed in terms of the Rodrigues' formula~\cite{mueller} 
as
\begin{align}\label{eq:gegenbauer}
	P_j(t) = \left(-\frac{1}{2}\right)^j \frac{\gfrac{n-1}{2}}{\gam{j + \frac{n-1}{2}}} (1-t^2)^{-\left(\frac{n-3}{2}\right)} \left(\frac{d}{dt}\right)^j (1-t^2)^{j + \frac{n-3}{2}}\ .
\end{align}
In particular, we will later employ the Funk-Hecke formula~\cite{mueller} which involves parameters $\lambda(n,\ell,j)$ as given in~\eqref{eq:oldLambda}. As these are rather unwieldy by itself, we simplify them using the following lemma.
\begin{lemma}\label{lem:simplifyLambda}
	Let $\ell, j,n \in \Natural$ and $n\geq 3$, then 
	\begin{align}\label{eq:oldLambda}
		\lambda(n,\ell,j) := \int_{-1}^{1} t^{2\ell} P_j(t) (1-t^2)^{\frac{n-3}{2}} dt
	\end{align}
	is given by
	\begin{align}
		\lambda(n,\ell,j) = \left\{
		\begin{array}{ll}
			\frac{\sqrt{\pi}}{2^{2\ell}} \frac{\gfrac{n-1}{2} \gam{2\ell + 1}}{\gam{\ell + 1 - \frac{j}{2}} \gam{\ell + \frac{n+j}{2}}} & \mbox{if } j \mbox{ is even and } j \leq 2\ell\ ,\\[2mm]
			0 & \mbox{if } j \mbox{ is odd or } j > 2\ell\ ,
		\end{array} \right.  \end{align}
	where $P_j(t)$ is the Gegenbauer polynomial given by the Rodrigues' formula of~\eqref{eq:gegenbauer}.
\end{lemma}
\begin{proof}
	Using integration by parts, it can be shown that~\cite[Lemma 11]{mueller}
	\begin{align}\label{eq:gegenStep1}
		\int_{-1}^{1} t^{2\ell} P_j(t) (1-t^2)^{\frac{n-3}{2}} dt = 
		\left(\frac{1}{2}\right)^j \frac{\gfrac{n-1}{2}}{\gam{j + \frac{n-1}{2}}}
		\int_{-1}^{1} \left(\frac{d}{dt}\right)^j(t^{2\ell}) (1-t^2)^{j + \frac{n-3}{2}} dt\ .
	\end{align}
	From $(d/dt)^j (t^{2\ell})$ we immediately see that $\lambda_j = 0$ for $j > 2\ell$.
	We may now simplify the right hand side, by noting that $(d/dt)^j(t^{2\ell}) = (2\ell)(2\ell-1)\ldots(2\ell - j + 1) t^{2\ell-j}$, to
	\begin{align}\label{eq:simplifiedLambda}
		\left(\frac{1}{2}\right)^j \frac{\gfrac{n-1}{2} \gam{2\ell + 1}}{\gam{j + \frac{n-1}{2}} \gam{2\ell + 1 - j}} 
		\int_{-1}^{1} t^{2\ell - j} (1-t^2)^{j + \frac{n-3}{2}} dt\ .
	\end{align}
	Let us now first consider the case when $j$ is odd.
	Note that in this case the function $f(t) := t^{2\ell - j}(1-t^2)^{j + \frac{n-3}{2}}$ obeys $f(-t) = - f(t)$. Hence, the integral vanishes and $\lambda_j = 0$.
	For the case when $j$ is even, we can use the fact that $f(t)$ is symmetric around the origin $t=0$ to write
	the integral in terms of the $\beta$-function to obtain (see e.g.~\cite[Equation 2.6]{decarli:norms} or~\cite[Section 10.4]{arfken})
	\begin{align}\label{eq:integralExpand}
		\int_{-1}^{1} t^{2\ell - j} (1-t^2)^{j + \frac{n-3}{2}} dt = \frac{\gam{\ell - \frac{j-1}{2}} \gam{j + \frac{n-1}{2}}}{\gam{\ell + \frac{n+j}{2}}}\ .
	\end{align}
	Combining~\eqref{eq:gegenStep1},~\eqref{eq:simplifiedLambda} and~\eqref{eq:integralExpand} we obtain
	\begin{align}
		\lambda(n,\ell,j) &= \left(\frac{1}{2}\right)^j \frac{\gfrac{n-1}{2} \gam{2\ell + 1}\gam{\ell - \frac{j-1}{2}}}{\gam{2\ell + 1 - j} \gam{\ell + \frac{n+j}{2}}}\ .
	\end{align}
	Noting that by Legendre's duplication formula we can write
	\begin{align}
		\frac{\gam{\ell - \frac{j-1}{2}}}{\gam{2\ell + 1 - j}} &= \frac{\gam{\ell - \frac{j-1}{2}}}{\gam{2\left(\ell - \frac{j-1}{2}\right)}}
		=\frac{\sqrt{\pi}}{2^{2\ell - j}} \frac{1}{\gam{\ell + 1 - \frac{j}{2}}}
	\end{align}
	then gives our claim.
\end{proof}

To form our approximations, we will need to compare ratios of the functions $\lambda(n,\ell,j)$ for larger values of $j$ to the one for $j=0$. Given our lemma above we obtain
the following corollary.

\begin{corollary}\label{cor:lambdaDiff}
	Let $\ell, j, n \in \Natural$, with $n \geq 3$ and $j$ even. Then
	\begin{align}
		\frac{\lambda(n,\ell,j)}{\lambda(n,\ell,0)} = \frac{\gam{\ell + 1}\gam{\ell + \frac{n}{2}}}{\gam{\ell + 1 - \frac{j}{2}} \gam{\ell + \frac{n+j}{2}}}\ .
	\end{align}
\end{corollary}

Whereas such ratios can of course be evaluated for any given values of $n$,$\ell$ and $j$, it is less obvious what their asymptotic behaviour (large $\ell$) would be. The following lemma derives a lower bound on said ratio, for which the assymptotics are very easy to understand. This will be used in the proof of our real De Finetti theorem. 

\begin{lemma}\label{lem:epsilonBound}
	Let $\ell, j, n \in \Natural$, $n \geq 3$, and let 
	$2\ell \geq j \geq 2$ be even. Then 
	\begin{align}
		1 - \frac{\lambda(n,\ell,j)}{\lambda(n,\ell,0)} \leq \frac{j\left(\frac{j+n}{2} - 1\right)}{2\ell + n}\ .
	\end{align}
\end{lemma}
\begin{proof}
	Our goal will be to lower bound $\frac{\lambda(n,\ell,j)}{\lambda(n,\ell,0)}$ for $j = 2k$ with $k \in \Natural$. Using Corollary~\ref{cor:lambdaDiff} we have
	\begin{align}\label{eq:firstEqLambda}
	\frac{\lambda(n,\ell,j)}{\lambda(n,\ell,0)} &= \frac{\gam{\ell + 1}\gam{\ell + \frac{n}{2}}}{\gam{\ell + 1 - k} \gam{\ell + \frac{n}{2} + k}} \ .
	\end{align}
	Using the fact that
	\begin{align}
		\frac{\gam{\ell + 1}}{\gam{\ell + \frac{n}{2} + k}} &= \frac{\ell (\ell-1)\ldots (\ell + 1 - k)\gam{\ell + 1 - k}}
		{\left(\ell + \frac{n}{2} + k - 1\right)\ldots\left(\ell + \frac{n}{2}\right) \gam{\ell + \frac{n}{2}}}
	\end{align}
	we can rewrite~\eqref{eq:firstEqLambda} as
	\begin{align}
		&\underbrace{\left(\frac{\ell + 1 - k}{\ell + \frac{n}{2}}\right)\left(\frac{\ell+2-k}{\ell + \frac{n}{2} + 1}\right)\ldots\left(\frac{\ell}{\ell + \frac{n}{2} + k - 1}\right)}_{ k \mbox{ factors}}
		\geq \left(\frac{\ell + 1 - k}{\ell + \frac{n}{2}}\right)^k\label{eq:lhs}\\
		&=\left(\frac{\ell + \frac{n}{2} - (k - 1 + \frac{n}{2})}{\ell + \frac{n}{2}}\right)^k
		\geq 1 - \frac{k(k + \frac{n}{2} - 1)}{\ell + \frac{n}{2}}
		=1 - \frac{j\left(\frac{j+n}{2} - 1\right)}{2\ell + n}\ ,
	\end{align}
	where the first inequality follows from the fact that for $n > 0$ and $k \geq 1$ we have $\ell + 1 - k \leq \ell + n/2$
	and hence $(\ell + 1 - k)/(\ell + n/2)$ is the smallest term on the left hand side of~\eqref{eq:lhs}. This can be seen 
	by noting that for any $x,y \geq 0$ the fact that $x \leq y$ implies $(x/y) (y+1)/(x+1) = (xy+x)/(xy+y)\leq 1$ and hence $x/y \leq (x+1)/(y+1)$.
\end{proof}

\begin{corollary}\label{cor:epsilonBound}
Let $\ell, j, n \in \Natural$, $n \geq 3$, and let 
	$2\ell \geq j \geq 2$ be even. Then 
	\begin{align}
		\frac{\lambda(n,\ell,0)}{\lambda(n,\ell,j)} -1\leq \frac{2j\left(\frac{j+n}{2} - 1\right)}{2\ell + n}=\epsilon(j,\ell,n)\ .
	\end{align}
	whenever $\epsilon(j,\ell,n)\leq1$.
\end{corollary}

\begin{proof}
Let $\alpha = 1- \lambda(n,\ell,j)/\lambda(n,\ell,0)$. Then Lemma~\ref{lem:epsilonBound} implies $\alpha\leq \epsilon(j,\ell,n)/2$. This gives us
\begin{align} \label{eq:newbound}
		\frac{\lambda(n,\ell,0)}{\lambda(n,\ell,j)} -1= \frac{1}{1-\alpha}-1\leq 2\alpha\leq \epsilon(j,\ell,n) \ .
	\end{align}
	Which is the required result. While the second inequality in (\ref{eq:newbound}) is just the result of Lemma~\ref{lem:epsilonBound}, the first inequality follows from the following small calculation
	\begin{eqnarray}
1+2\alpha -\frac{1}{1-\alpha} = \frac{(1+2\alpha)(1-\alpha)-1}{1-\alpha}= \frac{\alpha-2\alpha^2}{1-\alpha} \geq 0.
\end{eqnarray}
where the inequality holds when $\alpha\leq 1/2$ which holds due to Lemma~\ref{lem:epsilonBound} and the assumption that $\epsilon(j,\ell,n)\leq1$. We also require $\alpha\geq 0$ which follows from Lemma  \ref{lem:lambdaisincreasing} below.
\end{proof}

Let us now establish some further properties of the function $\lambda(n,\ell,j)$. 
\begin{lemma}\label{lem:lambdaisincreasing}
	Let $\ell, j, n \in \Natural$, and $n \geq 3$. Then for all even $j$ with
	$2\ell \geq j \geq 2$ the function $1/\lambda(n,\ell,j)$ is increasing in $j$. 
\end{lemma}
\begin{proof}
By~\eqref{eq:simplifiedLambda} we have
\begin{align}\label{eq:inverseLambda}
\frac{1}{\lambda(n,\ell,j)} = \frac{2^{2\ell}}{\sqrt{\pi}} \frac{\gam{\ell + 1 - \frac{j}{2}} \gam{\ell + \frac{n+j}{2}}}
{\gam{\frac{n-1}{2}} \gam{2\ell + 1}}\ .
\end{align}
Hence it suffices to show that $\gam{\ell + 1 - \frac{j}{2}} \gam{\ell + \frac{n+j}{2}}$ is increasing in $j$. 
Since $j = 2k$ is even, this is equivalent to showing that
\begin{align}\label{eq:oneFrac}
\frac{\gam{\ell + 1 - k}}{\gam{\ell - k}} \leq \frac{\gam{\ell + \frac{n}{2} + k + 1}}{\gam{\ell + \frac{n}{2} + k}}\ .
\end{align}
Note that $\gam{\ell + 1 - k}  = (\ell - k)\gam{\ell - k}$ and $\gam{\ell + n/2 + k + 1} = (\ell + n/2 + k)\gam{\ell + n/2 + k}$, hence~\eqref{eq:oneFrac}
is equivalent to 
\begin{align}
\ell - k  \leq \ell + \frac{n}{2} + k\ ,
\end{align}
which is our claim.
\end{proof}

We will also need to establish that the quantity $\|\cdot\|_{F1}$ does indeed form a norm for maximally symmetric matrices. 
\begin{claim}\label{claim:F1isnorm}
	The quantity $\|A\|_{F1}$ defines a norm on the space of maximally symmetric matrices $A \in \maxsymEll$.
\end{claim}
\begin{proof}
	Clearly, $\|cA\|_{F1} = c \|A\|_{F1}$ for any scalar $c$. To see that the triangle inequality holds note that
	\begin{align}
		\|A + B\|_{F1} &= \sup_F (\tr(Z_FA) + \tr(Z_FB))\\
		&\leq \sup_{F}\tr(Z_FA) + \sup_{F}\tr(Z_FB)
		= \|A\|_{F1} + \|B\|_{F1}\ .
	\end{align}
	To see that $\|A\|_{F1} = 0$ implies that $A = 0$ note that any $A\in \maxsymEll$ has a $P$ representation (Lemma~\ref{lem:hasPrepresentation})
	\begin{align}
		A = \sum_{jm} p^{A}_{jm} \sjm\ .
	\end{align}
	Furthermore, note that due to Lemmas~\ref{lem:PvsQ} and~\ref{lem:simplifyLambda} we can restrict the sum to be over $j \leq 2\ell$ and $j$ is even. 
	However for any $p^{A}_{jm} \neq 0$ choosing $F = s_{jm}(x)$ yields by Lemma~\ref{lem:sjmTrace} $\tr(Z_FA) \neq 0$. 
\end{proof}

\section{Relating optimal solutions of even and odd degree}
We will need the following purely technical lemma when dealing with homogeneous polynomials of odd degree. This will enable us to relate the optmization
problem of optimizing homogeneous polynomials of odd degree to the optimization of homogeneous polynomials of even degree.

\begin{lemma}\label{lem:cBound}
For all $\ell \in \Natural$
\begin{align}
\max_{c \geq 0} \frac{c^{2\ell - 1}}{(1+c^2)^{\ell}} = \frac{(2\ell -1)^{\ell - 1/2}}{\ell^\ell 2^\ell}\ .
\end{align}
\end{lemma}
\begin{proof}
Let $g(c,\ell) := c^{2\ell - 1}/(1+c^2)^{\ell}$.
Computing the derivative with respect to $c$ we have
\begin{align}
\frac{d g(c,\ell)}{d c} = - 2\ell \frac{c^{2\ell}}{(1+c^2)^{\ell+1}} + (2\ell - 1)\frac{c^{2\ell-2}}{(1+c^2)^\ell}\ .
\end{align}
A small calculation shows that this derivative vanishes $(=0)$ for $c = 0$ and for $c > 0$ iff
\begin{align}
c = \sqrt{2\ell - 1}\ .
\end{align}
We now show that this is the maximum. Note that for boundary point $c=0$ we have $g(0,\ell) = 0$. 
Computing the second derivative and evaluating it at the point $c = \sqrt{2\ell - 1}$ we obtain
\begin{align}
- \frac{(2\ell - 1)^{\ell - 1/2}}{2^\ell \ell^{\ell+1}}\ ,
\end{align}
which is clearly negative.
It remains to evaluate $g(c,\ell)$ at the maximum point which yields our claim.
\end{proof}

The following lemma will allow us to reduce the case of optimizing homogeneous polynomials of odd
degree over $S^{n-1}$, to the case of optimizing homogeneous polynomials of even degree over the $n+1$ dimensional sphere $S^{n}$.

\begin{lemma}\label{lem:turnIntoEven}
Let $T(x)$ be a homogeneous polynomial of odd degree $2\ell - 1$ in variables $x = (x_1,\ldots,x_n)$, and let $T'(x') = T(x) x_0$ for $x' = (x_0,x_1,\ldots,x_n)$, $x_0 \in \Real$.
Then
\begin{align}
\max_{x' \in S^{n}} T'(x')
 = \gamma(\ell) \max_{x \in S^{n-1}} T(x)\ ,
\end{align}
where $\gamma(\ell) = (2\ell - 1)^{\ell - 1/2}/(\ell^\ell 2^{\ell})$. 
\end{lemma}
\begin{proof}
Note that $T'(x')$ is a homogeneous polynomial of even degree $2\ell$. Define
\begin{align}
f(x') := \frac{T'(x')}{(\|x'\|_2^2)^\ell}\ .
\end{align}
Note that $\|x'\|_2^2$ is a homogeneous polynomial of degree $2\ell$. Since $T'(x')$ and $\|x'\|_2^2$ are homogeneous of degree $2\ell$, 
we have for all constants $k$
\begin{align}\label{eq:homoEquiv}
f(k x') = f(x')\ .
\end{align}
Observe that this implies that 
\begin{align}
\max_{x' \in \Real^{n+1}} f(x') = 
\max_{x' = (1,x_1,\ldots,x_n) \in \Real^{n+1}} f(x') =   
\max_{x' \in S^n} T'(x')\ ,
\end{align}
as we can choose the constant $k$ to normalize the solution of the l.h.s. to have $x_0 = 1$ or to reside on the hypersphere (where $\|x'\|_2 = 1$) without changing the optimum.
Note the expression in the middle can be rewritten as
\begin{align}\label{eq:maxFixedx0}
\max_{x' = (1,x_1,\ldots,x_n) \in \Real^{n+1}} f(x') = 
\max_{x \in \Real^{n}}\frac{T(x)}{(1 + \|x\|_2^2)^\ell}\ .
\end{align}
Optimizing over $x \in \Real^n$ is equivalent to optimizing over a vector $\hat{x} \in S^{n-1}$, and an optimization over the length of this vector
since we can write any $x$ as $x = c \hat{x}$ for some $\hat{x} \in S^{n-1}$ with $c = \|x\|_2$. 
Using the fact that $T(x)$ is homogeneous, we can hence express~\eqref{eq:maxFixedx0} as
\begin{align}
\max_{\hat{x} \in S^{n-1}} \max_{c \geq 0} \left(\frac{c^{2\ell-1}T(\hat{x})}{(1 + c^2)^\ell}\right) = 
\max_{\hat{x} \in S^{n-1}} T(\hat{x}) \underbrace{\left(\max_{c \geq 0} \frac{c^{2\ell-1}}{(1 + c^2)^\ell}\right)}_{= \gamma(\ell)}\ .
\end{align}
Our claim now follows by applying Lemma~\ref{lem:cBound} to evaluate $\gamma(\ell)$. 
\end{proof}

\end{document}